\documentclass[a4paper]{amsart}
\pdfoutput=1
\usepackage[utf8]{inputenc}
\usepackage[T1]{fontenc}
\usepackage{amssymb,mathtools}
\usepackage{mathrsfs}
\usepackage{dsfont}
\usepackage{lmodern}

\usepackage[all]{xy}
\usepackage{microtype}

\usepackage{enumitem}
\setlist[enumerate,1]{label=\textup{(\arabic*)}}% ensure enumerates in theorems are upright

\usepackage[lite]{amsrefs}
\newcommand*{\MRref}[2]{ \href{http://www.ams.org/mathscinet-getitem?mr=#1}{MR \textbf{#1}}}

\renewcommand*{\PrintDOI}[1]{\href{http://dx.doi.org/\detokenize{#1}}{doi: \detokenize{#1}}}

\BibSpec{book}{%
  +{}  {\PrintPrimary}                {transition}
  +{,} { \textit}                     {title}
  +{.} { }                            {part}
  +{:} { \textit}                     {subtitle}
  +{,} { \PrintEdition}               {edition}
  +{}  { \PrintEditorsB}              {editor}
  +{,} { \PrintTranslatorsC}          {translator}
  +{,} { \PrintContributions}         {contribution}
  +{,} { }                            {series}
  +{,} { \voltext}                    {volume}
  +{,} { }                            {publisher}
  +{,} { }                            {organization}
  +{,} { }                            {address}
  +{,} { \PrintDateB}                 {date}
  +{,} { }                            {status}
  +{}  { \parenthesize}               {language}
  +{}  { \PrintTranslation}           {translation}
  +{;} { \PrintReprint}               {reprint}
  +{.} { }                            {note}
  +{.} {}                             {transition}
  +{} { \PrintDOI}                   {doi}
  +{} { available at \url}            {eprint}
  +{}  {\SentenceSpace \PrintReviews} {review}
}

\usepackage[pdftitle={On C*-algebras associated to product systems},
pdfauthor={Camila F. Sehnem},
pdfsubject={Mathematics}
]{hyperref}

\numberwithin{equation}{section}

\theoremstyle{plain}
\newtheorem{thm}[equation]{Theorem}
\newtheorem{cor}[equation]{Corollary}
\newtheorem{lem}[equation]{Lemma}
\newtheorem{prop}[equation]{Proposition}

\theoremstyle{definition}
\newtheorem{defn}[equation]{Definition}
\newtheorem*{ack*}{Acknowledgements}

\theoremstyle{remark}
\newtheorem{rem}[equation]{Remark}
\newtheorem{example}[equation]{Example}

\newcommand{\ZZ}{\mathbb{Z}}

\newcommand{\NN}{\mathbb{N}}
\newcommand{\TT}{\mathbb{T}}
\newcommand{\FF}{\mathbb{F}}
\newcommand{\CC}{\mathbb{C}}

\newcommand*{\nb}{\nobreakdash}
\newcommand*{\Star}{\(^*\)\nobreakdash-}
\newcommand{\Cst}{\mathrm{C}^*}% C*-algebras and friends
\newcommand{\idealin}{\mathrel{\triangleleft}} % relation of being an ideal
\newcommand*{\Bound}{\mathbb B}%adjointable operators on a Hilbert module
\newcommand*{\Comp}{\mathbb K}%compact operators on a Hilbert module
\newcommand{\Hilm}[1][E]{\mathcal{#1}}% Hilbert module
\newcommand{\Toep}{\mathcal{T}}% Toeplitz algebra
\newcommand{\CP}{\mathcal{O}}% Cuntz-Pimsner algebra
\newcommand{\id}{\mathrm{id}}% identity
\DeclarePairedDelimiter{\bra}{\langle}{\rvert}% ket-bra notation
\DeclarePairedDelimiter{\ket}{\lvert}{\rangle}% ket-bra notation
\DeclarePairedDelimiterX{\braket}[2]{\langle}{\rangle}{#1\,\delimsize\vert\,\mathopen{}#2}% inner product
\DeclarePairedDelimiterX{\BRAKET}[2]{\langle}{\rangle}{\!\delimsize\langle#1\,\delimsize\vert\,\mathopen{}#2\delimsize\rangle\!}% inner product

\DeclarePairedDelimiterX{\setgiven}[2]{\{}{\}}{#1\,{:}\,\mathopen{}#2}% set given by

\begin{document}
\title[On \texorpdfstring{$\Cst$\nb-}{C*-}algebras associated to product systems]{On \texorpdfstring{$\Cst$\nb-}{C*-}algebras associated to product systems}

\author{Camila F. Sehnem}
\email{camila.fabre-sehnem@mathematik.uni-goettingen.de}

\address{Mathematisches Institut, Georg-August-Universität Göttingen,
  Bunsenstraße 3--5, 37073, Göttingen, Germany}


\begin{abstract}
Let $P$ be a unital subsemigroup of a group~$G$. We propose an approach to $\Cst$\nb-algebras associated to product systems over~$P$. We call the $\Cst$\nb-algebra of a given product system~$\Hilm$ its covariance algebra and denote it by~$A\times_{\Hilm}P$, where~$A$ is the coefficient $\Cst$\nb-algebra. We prove that our construction does not depend on the embedding~$P\hookrightarrow G$ and that a representation of~$A\times_{\Hilm}P$ is faithful on the fixed-point algebra for the canonical coaction of~$G$ if and only if it is faithful on~$A$. We compare this with other constructions in the setting of irreversible dynamical systems, such as Cuntz--Nica--Pimsner algebras, Fowler's Cuntz--Pimsner algebra, semigroup $\Cst$\nb-algebras of Xin Li and Exel's crossed products by interaction groups.
\end{abstract}
\maketitle

\section{Introduction}
 
Let $A$ be a $\Cst$\nb-algebra. A correspondence~$\Hilm\colon A\leadsto A$ consists of a right Hilbert $A$\nb-module with a nondegenerate left action of~$A$ implemented by a \Star homomorphism $\varphi\colon A\rightarrow\Bound(\Hilm)$. We say that~$\Hilm$ is faithful if the left action of~$A$ is injective. It is proper if~$A$ acts by compact operators on~$\Hilm$. A celebrated construction by Pimsner associates a $\Cst$\nb-algebra~$\CP_{\Hilm}$ to a faithful correspondence $\Hilm\colon A\leadsto A$~\cite{Pimsner:Generalizing_Cuntz-Krieger}. This is now known as a~\emph{Cuntz--Pimsner algebra}. It is the universal $\Cst$\nb-algebra for representations of~$\Hilm$ that satisfy a certain condition, now called \emph{Cuntz--Pimsner covariance}, on the ideal~$J=\varphi^{-1}(\Comp(\Hilm))\idealin A$. Pimsner's $\Cst$\nb-algebra includes many interesting $\Cst$\nb-algebras, such as crossed products by automorphisms and graph $\Cst$\nb-algebras for graphs with no sinks~\cite{Katsura:class_I}. It also covers crossed products by extendible and injective endomorphisms with hereditary range.

For a non-faithful correspondence~$\Hilm$, Pimsner's $\Cst$\nb-algebra may be zero. Muhly and Solel proposed a construction of $\Cst$\nb-algebras associated to (not necessarily faithful) correspondences by taking universal $\Cst$\nb-algebras for representations satisfying the covariance condition only on an ideal~$J\idealin A$ with~$J\subseteq\varphi^{-1}(\Comp(\Hilm))$~\cite{Muhly-Solel:Tensor}. In~\cite{Katsura:Cstar_correspondences}, Katsura provided necessary and sufficient conditions on the ideal~$J$ for the universal representation of~$\Hilm$ in~$\CP_{J,\Hilm}$ to be injective.  Inspired by graph $\Cst$\nb-algebras, among other constructions, he analysed the relative Cuntz--Pimsner algebra~$\CP_{ J_{\Hilm},\Hilm}$ with~$J_{\Hilm}\coloneqq (\ker\varphi)^\perp\cap\varphi^{-1}(\Comp(\Hilm))$. This~$\Cst$\nb-algebra has nice properties. First, the universal representation of~$\Hilm$ in~$\CP_{ J_{\Hilm},\Hilm}$ is injective. Hence, it encodes the correspondence structure of~$\Hilm$. In addition, $\CP_{ J_{\Hilm},\Hilm}$ satisfies a gauge-invariant uniqueness theorem, that is, a representation of~$\CP_{ J_{\Hilm},\Hilm}$ in a $\Cst$\nb-algebra~$B$ is faithful if it is faithful on the coefficient algebra~$A$ and respects the gauge action of the unit circle~$\TT$. 

 Roughly speaking, a product system may be regarded as an action of a semigroup by correspondences over a $\Cst$\nb-algebra. A product system over a semigroup~$P$ with unit element denoted by~$e$ is a family of correspondences~$\Hilm=(\Hilm_p)_{p\in P}$ with~$\Hilm_e=A$ together with correspondence isomorphisms $\Hilm_p\otimes_A\Hilm_q\cong \Hilm_{pq}$ subject to certain axioms. It is equivalent to a single correspondence when the underlying semigroup is~$\NN$~\cite{Albandik-Meyer:Colimits}. Product systems were introduced in this context by Fowler in~\cite{Fowler:Product_systems}, following the work of Arveson on continuous product systems of Hilbert spaces developed in~\cite{Arveson:Continuous_Fock}. As for single correspondences, examples of product systems arise naturally from semigroups of endomorphisms~\cite{Fowler:Product_systems,Larsen:Crossed_abelian}.
 
 Fowler defined the Toeplitz algebra~$\Toep_{\Hilm}$ of a given product system~$\Hilm$ as the universal $\Cst$\nb-algebra for representations of~$\Hilm$, thus generalising Toeplitz algebras of single correspondences. However, unlike the case of single correspondences, the Toeplitz algebra of a product system is in general too big. For example, the universal $\Cst$\nb-algebra for representations of the trivial bundle over~$\NN\times\NN$ is not nuclear\footnote{A $\Cst$\nb-algebra~$A$ is \emph{nuclear} if for every $\Cst$\nb-algebra~$B$ there exists a unique~$\Cst$\nb-norm on the tensor product~$A\odot B$.}(see~\cite{murphy1996}). This is precisely the universal $\Cst$\nb-algebra generated by two commuting isometries. Fowler also constructed the Cuntz--Pimsner algebra of a product system~$\Hilm=(\Hilm_p)_{p\in P}$ as the universal $\Cst$\nb-algebra for representations that are Cuntz--Pimsner covariant on~$J_p=\varphi_p^{-1}(\Comp(\Hilm_p))$ for all~$p\in P$. As in Pimsner's original construction, Fowler's Cuntz--Pimsner algebra might be trivial if~$\Hilm$ is non-faithful.

Following the work on Toeplitz algebras associated to quasi-lattice ordered groups by Nica~\cite{Nica:Wiener--hopf_operators} and also Laca and Raeburn \cite{LACA1996415}, Fowler introduced and studied in~\cite{Fowler:Product_systems} the class of \emph{compactly aligned} product systems over semigroups arising  from quasi-lattice orders.  That is, $P$ is a subsemigroup of a group~$G$ and~$(G,P)$ is a quasi-lattice ordered group in the sense of Nica \cite{Nica:Wiener--hopf_operators}. For instance, a $k$\nb-graph gives rise to a compactly aligned product system over $\NN^k$ precisely when it is finitely aligned~\cite[Theorem 5.4]{Raeburn-Sims:Product_graphs}. Fowler~\cite{Fowler:Product_systems} built a $\Cst$\nb-algebra out of a compactly aligned product system that is universal for representations satisfying an extra condition, called \emph{Nica covariance}. The resulting $\Cst$\nb-algebra, known as a \emph{Nica--Toeplitz algebra}, is spanned by elements of the form $t(\Hilm_p)t(\Hilm_q)^*$ with $p,q$ in~$P$. Hence, it is much more tractable than the usual Toeplitz algebra~$\mathcal{T}_{\Hilm}$. For amenable systems, Fowler was able to characterise faithful representations of this algebra~\cite[Theorem 7.2]{Fowler:Product_systems}. Also under an amenability assumption, a result deriving nuclearity for a Nica--Toeplitz algebra from nuclearity of the underlying coefficient algebra was established in~\cite[Theorem 6.3]{Rennie_groupoidfell}.

However, the problem of finding a $\Cst$\nb-algebra that approximates the structure of a given compactly aligned product system in an optimal way has not been completely solved. The question is: for a compactly aligned product system~$\Hilm=(\Hilm_p)_{p\in P}$, which quotient of the Nica--Toeplitz algebra~$\mathcal{N}\Toep_{\Hilm}$ gives in an appropriate sense the smallest $\Cst$\nb-algebra so that the representation of~$\Hilm$ in the corresponding quotient remains injective? Under an amenability assumption, such a $\Cst$\nb-algebra would be a co-universal object in the sense of~\cite{Carlsen-Larsen-Sims-Vittadello:Co-universal} for gauge-compatible Nica covariant representations of~$\Hilm$.

Answering the above question was the main objective of the work of Sims and Yeend in~\cite{Sims-Yeend:Cstar_product_systems}. They were able to associate a $\Cst$\nb-algebra~$\mathcal{NO}_{\Hilm}$ to a given compactly aligned product system~$\Hilm$, called \emph{Cuntz--Nica--Pimsner algebra}, so that the universal representation of~$\Hilm$ in~$\mathcal{NO}_{\Hilm}$ is injective for a large class of product systems~\cite[Theorem 4.1]{Sims-Yeend:Cstar_product_systems}. This is a quotient of the Nica--Toeplitz algebra of~$\Hilm$. Their notion of covariant representations is more technical than the usual Cuntz--Pimsner covariance since it involves additional relations. Sims and Yeend proved that Cuntz--Nica--Pimsner algebras include Cuntz--Krieger algebras of finitely aligned higher-rank graphs~\cite[Proposition 5.4]{Sims-Yeend:Cstar_product_systems} and Katsura's relative Cuntz--Pimsner algebras of single correspondences~\cite[Proposition 5.3]{Sims-Yeend:Cstar_product_systems}. The analysis of co-universal properties for these algebras was provided in~\cite{Carlsen-Larsen-Sims-Vittadello:Co-universal}. If either~$\Hilm$ is faithful or the representation of~$\Hilm$ in~$\mathcal{NO}_{\Hilm}$ is injective and~$P$ is directed, then under an amenability assumption a gauge-compatible representation of~$\mathcal{NO}_{\Hilm}$ is faithful if and only if it is faithful on~$A$ \cite[Corollary 4.11]{Carlsen-Larsen-Sims-Vittadello:Co-universal}.

Even though the universal representation of a compactly aligned product system~$\Hilm$ in~$\mathcal{NO}_{\Hilm}$ is injective for many examples, it might fail to be faithful even for proper product systems over totally ordered semigroups such as the positive cone of~$\ZZ\times\ZZ$ with the lexicographic order~\cite[Example 3.16]{Sims-Yeend:Cstar_product_systems}. In addition, \cite[Example 3.9]{Carlsen-Larsen-Sims-Vittadello:Co-universal} shows that if~$P$ is not directed, a representation of~$\mathcal{NO}_{\Hilm}$ that is faithful on~$A$ need not be faithful even for an amenable system. Moreover, Cuntz--Nica--Pimsner algebras cannot handle product systems over semigroups that are not positive cones of quasi-lattice orders.

In this paper, we let~$P$ be a subsemigroup of a group~$G$ and construct a $\Cst$\nb-algebra~$A\times_{\Hilm}P$ from a product system~$\Hilm=(\Hilm_p)_{p\in P}$ satisfying the conditions~(A) and (B) of~\cite{Sims-Yeend:Cstar_product_systems}: the universal representation of~$\Hilm$ in~$A\times_{\Hilm}P$ is faithful and a representation of~$A\times_{\Hilm}P$ is faithful on the fixed-point algebra for the canonical gauge coaction of~$G$ if and only if it is faithful on~$A$. To do so, we look at the topological $G$\nb-grading~$\{\Toep_{\Hilm}^g\}_{g\in G}$ of the Toeplitz algebra of~$\Hilm$ coming from the canonical coaction of~$G$. Inspired by the notion of Cuntz--Nica--Pimsner covariance introduced by Sims and Yeend, we analyse a class of representations of~$\Toep^e_{\Hilm}$ coming from quotients of the usual Fock representation of~$\Hilm$ on~$\bigoplus_{\substack{p\in P}}\Hilm_p$. Following ideas of~\cite{Exel:New_look}, we use the directed set consisting of finite subsets of~$G$ to define what we call \emph{strong covariance}. Although it explicitly involves elements of~$G$, this notion of covariance does not depend on the embedding~$P\hookrightarrow G$. In other words, different groups containing~$P$ as a subsemigroup produce the same quotient of~$\Toep_{\Hilm}$. We refer to the universal $\Cst$\nb-algebra $A\times_{\Hilm}P$ for strongly covariant representations of~$\Hilm$ as its covariance algebra.

The notion of covariance introduced here is technical and in general difficult to verify. However, we present an equivalent and considerably simpler definition of strong covariance for compactly aligned product systems over quasi-lattice ordered groups. We show that~$A\times_{\Hilm}P$ coincides with the Cuntz--Nica--Pimsner algebra of~$\Hilm$ if either~$P$ is directed and the canonical representation of~$\Hilm$ in $\mathcal{NO}_{\Hilm}$ is injective or $\Hilm$ is faithful. This is precisely the hypothesis of~\cite[Proposition 3.7]{Carlsen-Larsen-Sims-Vittadello:Co-universal}. 

We prove that our construction includes Fowler's Cuntz--Pimsner algebra if~$\Hilm$ is a proper and faithful product system over a cancellative Ore monoid. Again only assuming that~$P$ is embeddable into a group, we construct a product system~$\Hilm$ as in~\cite[Section 5]{Albandik-Meyer:Product} so that~$A\times_{\Hilm}P$ recovers the semigroup $\Cst$\nb-algebra of Xin Li whenever the family of constructible right ideals of~$P$ is independent (see~\cite[Definition 2.26]{Li:Semigroup_amenability}). In general, the covariance algebra of such a product system corresponds to the semigroup $\Cst$\nb-algebra~${\Cst}^{(\cup)}_s(P)$ in the notation of~\cite{Li:Semigroup_amenability}. We also assume that~$P$ is a reversible cancellative semigroup and describe a class of Exel's crossed products by interaction groups as covariance algebras. Thus our approach may inspire further $\Cst$\nb-constructions for irreversible dynamical systems.

This paper is organised as follows. In Section \ref{sec:background}, we recall some basic definitions and describe the coaction of a group~$G$ containing~$P$ on the Toepliz algebra associated to a product system $\Hilm=(\Hilm_p)_{p\in P}$ as well as the corresponding topological $G$\nb-grading of~$\Toep_{\Hilm}$. In Section~\ref{sec:covariance}, we present our main theorem, namely, Theorem~\ref{thm:cuntz_pimsner_algebra}. We first define a certain gauge-invariant ideal~$J^ G_{\infty}$ of~$\Toep_{\Hilm}$. The corresponding quotient~$\Toep_{\Hilm}/J^G_{\infty}$ becomes the object of study. We show that~$A$ embeds into this quotient. Moreover, it carries a canonical topological $G$\nb-grading and also satisfies an analogue of condition~(B) mentioned previously. Applying this fact to the universal group of~$P$, we are able to show that such a quotient of~$\Toep_{\Hilm}$ is independent of the choice of the group containing~$P$ as a subsemigroup. The remaining facts to be proved are in Theorem \ref{thm:cuntz_pimsner_algebra}. We refer to this quotient of~$\Toep_{\Hilm}$ as the covariance algebra of~$\Hilm$.

In Section \ref{sec:examples}, we illustrate our construction by comparing it with other $\Cst$\nb-algebras arising from irreversible dynamical systems. We have included an appendix concerned with the topological grading coming from a discrete coaction.

\section{Product systems}\label{sec:background}
We recall some basic notions and results that will be needed in the sequel.
\subsection{Notation and basic notions} Let~$P$ be a semigroup with identity~$e$ and~$A$ a $\Cst$\nb-algebra. A \emph{product system} over $P$ of $A$\nb-correspondences consists of:
\begin{enumerate}
\item[(i)] a correspondence $\Hilm_p\colon A\leadsto A$ for each $p\in P\setminus\{e\}$;
\item[(ii)] correspondence isomorphisms $\mu_{p,q}\colon\Hilm_p\otimes_A\Hilm_q\overset{\cong}{\rightarrow}\Hilm_{pq}$, also called \emph{multiplication maps}, for all $p,q\in P\setminus\{e\}$;
\end{enumerate}

In addition, we let $\Hilm_e=A$ with the obvious structure of correspondence over~$A$. The multiplication maps $\mu_{e,p}$ and $\mu_{p,e}$ implement the left and right actions of~$A$ on~$\Hilm_p$, respectively, so that~$\mu_{e,p}(a\otimes\xi_p)=\varphi_p(a)\xi_p$ and $\mu_{p,e}(\xi_p\otimes a)=\xi_pa$ for all~$a\in A$ and $\xi_p\in\Hilm_p$. 

 This data must make the following diagram commute:
  \[
  \xymatrix{
    (\Hilm_p\otimes_A\Hilm_q)\otimes_A\Hilm_r  \ar@{->}[d]^{\mu_{p,q}\otimes1}
   \ar@{<->}[rr]& &     \Hilm_p\otimes_A(\Hilm_q\otimes_A\Hilm_r)
   \ar@{->}[rr]^{1\otimes\mu_{q,r}}&&
    \Hilm_p\otimes_A\Hilm_{qr} \ar@{->}[d]^{\mu_{p,qr}} \\
    \Hilm_{pq}\otimes_A\Hilm_r   \ar@{->}[rrrr]^{\mu_{pq,r}}&& &&
    \Hilm_{pqr}  
.  }
  \]
  
  A product system $\Hilm=(\Hilm_p)_{p\in P}$ will be called \emph{faithful} if $\varphi_p$ is injective for all $p\in P$. It is \emph{proper} if $A$ acts by compact operators on~$\Hilm_p$ for all~$p$ in~$P$. 
  
  \begin{defn}A \emph{Toeplitz representation} of $\Hilm=(\Hilm_p)_{p\in P}$ in a $\Cst$\nb-algebra $B$ consists of linear maps~$\psi_p\colon\Hilm_p\rightarrow B$, for all $p\in P\setminus\{e\},$ and a \Star homomorphism $\psi_e\colon A\rightarrow B$, satisfying the following two axioms:
  \begin{enumerate}
  \item[(T1)] $\psi_p(\xi)\psi_q(\eta)=\psi_{pq}(\xi\eta)$ for all $p,q\in P$, $\xi\in\Hilm_p$ and  $\eta\in\Hilm_q$;
  \item[(T2)] $\psi_p(\xi)^*\psi_p(\eta)=\psi_e(\braket{\xi}{\eta}),$ for all $p\in P$ and $\xi, \eta\in\Hilm_p$.  
  \end{enumerate}  
  \end{defn}
  
  The \emph{Toeplitz algebra} of~$\Hilm$, denoted by $\Toep_{\Hilm}$, is the universal $\Cst$\nb-algebra for Toeplitz representations of $\Hilm$ \cite{Fowler:Product_systems, Pimsner:Generalizing_Cuntz-Krieger}. 
  
 \subsection{The coaction on the Toeplitz algebra} Let~$G$ be a discrete group. Let $\delta_G$ be the \Star homomorphism $\Cst(G)\rightarrow \Cst(G)\otimes \Cst(G)$ defined by~$\delta_G(u_g)=u_g\otimes u_g$, where~$u_g$ denotes the image of~$g\in G$ under the canonical unitary representation of~$G$ in~$\Cst(G))$. A \emph{full coaction} of $G$ on a $\Cst$\nb-algebra~$A$ is a nondegenerate and injective \Star homomorphism $\delta\colon A\rightarrow A\otimes \Cst(G)$ such that $$(\delta\otimes\id_{\Cst(G)})\delta=(\id_A\otimes\delta_G)\delta.$$ The triple $(A,G,\delta)$ is referred to as a coaction. See, for instance, \cite[Definition A.21]{Echterhoff-Kaliszewski-Quigg-Raeburn:Categorical} and also~\cite{Quigg:Discrete_coactions_and_bundles}. Replacing~$\Cst(G)$ by~$\Cst_r(G)$ and adapting the coaction identity accordingly, we obtain what is called a \emph{reduced} coaction~\cite{Quigg:FullAndReducedCoactions}. Here we will only use full coactions. So we will omit the term ``full''. 
  
If~$(A,G,\delta)$ is a coaction, then $\delta$ provides~$B$ with a topological $G$\nb-grading for which the spectral subspace at $g\in G$ is given by $A_g=\{a\in A\mid \delta(a)=a\otimes u_g\}$. We refer to~$A_e$ as the \emph{fixed-point algebra} for the coaction of~$G$ on~$A$.

 The idea of considering coactions on Toeplitz algebras associated to product systems goes back to~\cite[Proposition 4.7]{Fowler:Product_systems} and also~\cite{Carlsen-Larsen-Sims-Vittadello:Co-universal} for Nica--Toeplitz algebras. Given a product system~$\Hilm=(\Hilm_p)_{p\in P}$ and a discrete group~$G$ with~$G\supseteq P$, we will need the topological $G$\nb-grading coming from the canonical coaction of~$G$ on the Toeplitz algebra~$\Toep_{\Hilm}$ in the subsequent section. Hence we begin with a description of such a coaction and the corresponding $G$\nb-grading of~$\Toep_{\Hilm}$.

Let $\Hilm=(\Hilm_p)_{p\in P}$ be a product system. Suppose that~$P$ is a subsemigroup of a group~$G$. There is a representation of~$\Hilm$ in $\Toep_{\Hilm}\otimes \Cst(G)$ which sends~$\xi_p\in\Hilm_p$ to $\tilde{t}(\xi_p)\otimes u_p$. By the universal property of~$\Toep_{\Hilm}$, this yields a \Star homomorphism $\widetilde{\delta}\colon\Toep_{\Hilm}\rightarrow\Toep_{\Hilm}\otimes \Cst(G)$.

\begin{lem}\label{lem:grading_description} The above \Star homomorphism $\widetilde{\delta}\colon\Toep_{\Hilm}\to\Toep_{\Hilm}\otimes \Cst(G)$ provides a full coaction of~$G$ on~$\Toep_{\Hilm}$. Moreover, the spectral subspace~$\Toep_{\Hilm}^g$ at~$g\in G$ associated to $\widetilde{\delta}$ is the closure of sums of elements of the form $$\tilde{t}(\xi_{p_1})\tilde{t}(\xi_{p_2})^*\ldots\tilde{t}(\xi_{p_{n-1}})\tilde{t}(\xi_{p_n})^*,$$ where $n\in \NN$, $p_1p_2^{-1}\ldots p_{n-1}p_n^{-1}=g$ and $\xi_{p_i}\in \Hilm_{p_i}$ for all $i\in\{1,2,\ldots,n\}$.
\end{lem} 
\begin{proof} We begin by proving that $\widetilde{\delta}$ is nondegenerate.  Let $(u_{\lambda})_{\lambda\in\Lambda}$ be an approximate identity for~$A$. For each $p\in P$, both the left and right actions of~$A$ on~$\Hilm_p$ are nondegenerate. Consequently, $\big(\widetilde{t}_e(u_{\lambda})\big)_{\lambda\in\Lambda}$ is an approximate unit for~$\Toep_{\Hilm}$. Hence its image under~$\widetilde{\delta}$ satisfies, for all $b\in\Toep_{\Hilm}$ and $g\in G$, $$\lim_{\substack{\lambda}}\widetilde{\delta}(\widetilde{t}_e(u_{\lambda}))(b\otimes u_g)=\lim_{\substack{\lambda}}\,(\widetilde{t}_e(u_{\lambda})\otimes u_e)(b\otimes u_g)=\lim_{\substack{\lambda}}\widetilde{t}_e(u_{\lambda})b\otimes u_g=b\otimes u_g.$$ This guarantees that $\widetilde{\delta}$ is nondegenerate. In addition, for all $p\in P$, we have $$(\widetilde{\delta}_p\otimes\id_{\Cst(G)})\widetilde{\delta}_p=(\id_{\Toep_{\Hilm}}\otimes\delta_G)\widetilde{\delta}_p$$ on $\widetilde{t}(\Hilm_p)$. Thus $\widetilde{\delta}$ satisfies the coaction identity on~$\Toep_{\Hilm}$ as well, because it is generated by~$\widetilde{t}(\Hilm)$ as a $\Cst$\nb-algebra.

It remains to prove that $\widetilde{\delta}$ is injective. Indeed, let $1_G\colon G\to\CC,$ $g\mapsto 1$ be the trivial group homomorphism. Then $(\id_{\Toep_{\Hilm}}\otimes1_G)\circ\widetilde{\delta}=\id_{\Toep_{\Hilm}}$ if we identify~$\Toep_{\Hilm}$ with~$\Toep_{\Hilm}\otimes\CC$ in the canonical way. So~$\widetilde{\delta}$ is injective. Hence~$\widetilde{\delta}$ is a full coaction of~$G$ on the Toeplitz algebra of~$\Hilm$. 

 Now let~$\Toep^g_{\Hilm}$ be the spectral subspace at~$g\in G$ for~$\widetilde{\delta}$ and let~$\widetilde{\delta}_g$ denote the projection of~$\Toep_{\Hilm}$ onto~$\Toep_{\Hilm}^g$ as in \cite[Lemma 1.3]{Quigg:Discrete_coactions_and_bundles} (see also \cite[Corollary 19.6]{Exel:Partial_dynamical}). Take~$b$ in~$\Toep^g_{\Hilm}$. Since~$\widetilde{\delta}_g$ is contractive and~$\Toep_{\Hilm}$ is generated by~$\tilde{t}(\Hilm)$ as a $\Cst$\nb-algebra, we may suppose that $$b=\overset{m}{\sum_{\substack{j=1}}}\,\tilde{t}(\xi_{p^j_{1}})\tilde{t}(\xi_{p^j_{2}})^*\ldots\tilde{t}(\xi_{p^j_{k_j-1}})\tilde{t}(\xi_{p^j_{k_j}})^*,$$ where $m, k_j\in\NN$ for all $j$ in $\{1,2,\ldots,m\}$ and $\xi_{p^j_{i}}\in \Hilm_{p^j_i}$. The assertion then follows from the fact that~$\widetilde{\delta}_g$ vanishes on any element of the form $$\tilde{t}(\xi_{p_1})\tilde{t}(\xi_{p_2})^*\ldots\tilde{t}(\xi_{p_{n-1}})\tilde{t}(\xi_{p_n} )^*$$ with $p_1{p_2}^{-1}\ldots p_{n-1}p_{n}^{-1}\neq g$.
\end{proof}

We will refer to the coaction obtained in the previous lemma as the \emph{generalised gauge coaction} of~$G$ on~$\Toep_{\Hilm}$.

\section{\texorpdfstring{$\Cst$\nb-}{C*-}algebras associated to product systems}\label{sec:covariance}

In this section, we combine ideas of Exel and Sims and Yeend (see \cite{Exel:New_look,Sims-Yeend:Cstar_product_systems}) to construct a $\Cst$\nb-algebra~$A\times_{\Hilm}P$ out of a product system~$\Hilm$ so that a representation of~$A\times_{\Hilm}P$ is faithful on its fixed-point algebra for the canonical coaction of a group containing~$P$ if and only if it is faithful on the coefficient algebra. Our results apply to product systems over semigroups that can be embedded into groups.

\subsection{Strongly covariant representations} We first introduce the notion of strongly covariant representations. Let~$P$ be a semigroup with unit~$e$. Assume that~$P$ is embeddable into a group. That is, there is a group~$G$ and an injective semigroup homomorphism~$\gamma\colon P\to G$. Fix a $\Cst$\nb-algebra~$A$ and a product system~$\Hilm=(\Hilm_p)_{p\in P}$ over~$A$.

 Let $F\subseteq G$ be a finite subset. We set $$K_F\coloneqq\underset{g\in F}{\bigcap}gP.$$ So $K_{\{e,g\}}\neq\emptyset$ if and only if~$g$ may be written as~$pq^{-1}$ for some $p,q\in P$. In addition, $K_{gF}=gK_{F}$ for all~$g\in G$, where~$$gF=\left\{gh\middle|\,h\in F\right\}.$$ If~$p\in P$ and~$p\in K_{\{p,g\}}$, then~$p=gq$ for some~$q\in P$, which implies~$g=pq^{-1}$. 
 
For each~$p\in P$ and each~$F\subseteq G$ finite, we define an ideal~$I_{p^{-1}(p\vee F)}\idealin A$ as follows. Given 
$g\in F$, we let
$$I_{p^{-1}K_{\{p,g\}}}\coloneqq\begin{cases}\underset{r\in K_{\{p,g\}}}{\bigcap}\ker\varphi_{p^{-1}r}&\text{if } K_{\{p,g\}}\neq\emptyset\text{ and }p\not\in K_{\{p,g\}},\\
A& \text{otherwise. }
\end{cases}$$ 
We then let $$I_{p^{-1}(p\vee F)}\coloneqq\underset{g\in F}{\bigcap}I_{p^{-1}K_{\{p,g\}}}.$$  This gives a new correspondence $\Hilm_{F}\colon A\leadsto A$ by setting \begin{equation}\label{eq:finite_set_correspondence}\Hilm_{F}\coloneqq\bigoplus_{\substack{p\in P}}\Hilm_pI_{p^{-1}(p\vee F)}.
\end{equation}

Finally, let $\Hilm^+_F$ denote the right Hilbert $A$\nb-module $\bigoplus_{\substack{g\in G}}\Hilm_{gF}$. For each $\xi\in \Hilm_p$, we define an operator $t_F^p(\xi)\in\Bound(\Hilm_F^+)$ so that it maps the direct summand~$\Hilm_{gF}$ into~$\Hilm_{pgF}$ for all~$g\in G$. Explicitly, $$t_F^p(\xi)(\eta_{r})\coloneqq\mu_{p,r}(\xi\otimes_A\eta_{r}),\quad\eta_{r}\in\Hilm_{r}I_{r^{-1}(r\vee gF)}.$$ This is well defined because $I_{r^{-1}(r\vee F)}=I_{(pr)^{-1}(pr\vee pF)}$ for each
$F\subseteq G$ finite and each~$p\in P$. Its adjoint $t_F^p(\xi)^*$ sends $\mu_{p,r}(\zeta_p\otimes \eta_r)$ to $\varphi_r(\braket{\xi}{\zeta_p})\eta_r$. This is well defined because~$I_{s^{-1}(s\vee F)}=I_{s^{-1}p(p^{-1}s\vee p^{-1}F)}$ for all~$s\in pP$.  This gives a representation~$t_F=\{t_F^p\}_{p\in P}$ of~$\Hilm$ and hence a \Star homomorphism $\Toep_{\Hilm}\rightarrow\Bound(\Hilm_F^+),$ which we still denote by~$t_F$.

Let us denote by $Q^F_g$ the projection of $\Hilm_F^+$ onto the direct summand $\Hilm_{gF}$. Then $$t_F^p(\xi)Q^F_g=Q^F_{pg}t_F^p(\xi),\quad t_F^p(\xi)^*Q^F_g=Q^F_{p^{-1}g}t_F^p(\xi)^*$$ for all $p\in P$. Set $$\Toep^{e,F}_{\Hilm}\coloneqq Q^F_et_F(\Toep_{\Hilm}^e)Q^F_e,$$ where~$\Toep_{\Hilm}^e$ is the fixed-point algebra of~$\Toep_{\Hilm}$ for the gauge coaction of~$G$. If~$F_1\subseteq F_2$ are finite subsets of~$G$, then $$I_{p^{-1}(p\vee F_1)}\supseteq I_{p^{-1}(p\vee F_2)}$$ for all~$p\in P$. Hence~$\Hilm_{F_2}$ may be regarded as a closed submodule of~$\Hilm_{F_1}$. The restriction of~$Q^{F_1}_et_{F_1}(\Toep_{\Hilm}^e)Q^{F_1}_e$ to~$\Hilm_{F_2}$ gives a \Star homomorphism $$t_{F_1,F_2}\colon\Toep^{e,F_1}_{\Hilm}\rightarrow\Toep^{e,F_2}_{\Hilm}$$ such that~$\left(t_{F_1,F_2}\circ t_{F_1}\right)(b)=t_{F_2}(b)$ on~$\Hilm_{F_2}$ for all~$b\in\Toep_{\Hilm}^e$. For $F_1\subseteq F_2\subseteq F_3$, we have $t_{F_2,F_3}\circ t_{F_1,F_2}=t_{F_1,F_3}$. So we let~$F$ range in the directed set determined by all finite subsets of~$G$ and define an ideal~$J_e\idealin\Toep_{\Hilm}^e$ by $$J_e\coloneqq\left\{b\in\Toep_{\Hilm}^e\middle|\;\lim_{\substack{F}}\|b\|_F=0\right\},$$ where $\|b\|_F\coloneqq \|Q_F^et_F(b)Q_F^e\|$. We are now ready to introduce our notion of strong covariant representations.

\begin{defn}\label{def:strong_covariance} We will say that a representation of~$\Hilm$ is \emph{strongly covariant} if it vanishes on~$J_e$.

Let $J_\infty\idealin\Toep_{\Hilm}$ be the ideal generated by~$J_e$. Then~$\Toep_{\Hilm}\big/J_{\infty}$ is universal for strongly covariant representations of~$\Hilm$.
\end{defn}

The idea behind~\eqref{eq:finite_set_correspondence} started from the realisation that the correspondences~$\widetilde{\Hilm}_p$'s built in \cite{Sims-Yeend:Cstar_product_systems} out of~$\Hilm$ could be replaced by the~$\Hilm_F$'s in order to give the same notion of covariant representations if~$\Hilm$ is compactly aligned and $\widetilde{\phi}$\nb-injective and~$P$ is directed.  This is shown in Proposition~\ref{prop:Cuntz-Nica-Pimsner}. In this case, it suffices to consider finite subsets of~$P$ because~$(G,P)$ is quasi-lattice ordered. Exel constructed a $\Cst$\nb-algebra out of a nondegenerate interaction group~$(A,G,V)$ with the property that a representation of this crossed product is faithful on the fixed-point algebra for the canonical coaction of~$G$ if and only if it is faithful on~$A$. To show that~$A$ embeds into the crossed product, he built a faithful covariant representation by using inductive limits over finite subsets of~$G$ (see \cite[Section 9]{Exel:New_look}). This is related to product systems because, in fact, the main purpose in~\cite{Exel:New_look} was to introduce a new notion of crossed products by semigroups of unital and injective endomorphisms which can be enriched to interaction groups. Here we want to associate a $\Cst$\nb-algebra to a product system~$\Hilm=(\Hilm_p)_{p\in P}$ with the property that a representation of this resulting $\Cst$\nb-algebra is faithful on the fixed-point algebra for the canonical coaction of a group containing~$P$ if and only if it is faithful on~$A$. To achieve this goal, we believe its unit fibre should be a direct limit of $\Cst$\nb-algebras with injective connecting maps (see~\cite{Albandik-Meyer:Product,Exel:New_look,Kwasniewski-Szymanski:Ore,Pimsner:Generalizing_Cuntz-Krieger}), although in general this fact is not established. So, combining all these ideas and modifying the Cuntz--Nica--Pimsner covariance condition accordingly, we arrived at the~$\Hilm_F$'s and Definition~\ref{def:strong_covariance}.

Our next immediate goal is to prove that $A$ embeds into the quotient~$\Toep_{\Hilm}\big/J_{\infty}$.

\begin{lem}\label{lem:ideal_description} The ideal $J_{\infty}$ coincides with~$\overline{\underset{g\in G}{\bigoplus}}\Toep^g_{\Hilm}J_e.$ As a consequence, $$J_{\infty}=\overline{\underset{g\in G}{\bigoplus}}(J_{\infty}\cap\Toep^g_{\Hilm}).$$
\end{lem}
\begin{proof} In order to prove the first assertion, it suffices to show that $J_e\Toep^g_{\Hilm}\subseteq\Toep^g_{\Hilm}J_e$ for all~$g\in G$. To do so, let~$b\in J_e$ and~$0\neq c_g\in\Toep^g_{\Hilm}$. Let~$\varepsilon>0$. There is~$F\subseteq G$ finite with $\|b\|_S<\dfrac{\varepsilon}{\|c_g\|}$ for all finite subsets~$S$ of~$G$ with~$S\supseteq F$ because $b\in J_e$. Set $$F'\coloneqq g^{-1}F=\{g^{-1}h\mid h\in F\}.$$ Since $c_g$ maps $\Hilm_{F'}$ into $\Hilm_{gF'}=\Hilm_F$, it follows that~$\|c_g^*b^*bc_g\|_{F'}<\varepsilon^2$.
This guarantees that $$(J_e\Toep^g_{\Hilm})^*(J_e\Toep^g_{\Hilm})\subseteq J_e.$$ Hence~$J_e\Toep^g_{\Hilm}\subseteq\Toep^g_{\Hilm}J_e$ (see, for example, \cite[Lemma 3.5]{Pimsner:Generalizing_Cuntz-Krieger}). Applying the first assertion and the continuity of the projection of~$\Toep_{\Hilm}$ onto~$\Toep^g_{\Hilm}$, we deduce that $J_{\infty}\cap\Toep^g_{\Hilm}=\Toep^g_{\Hilm}J_e$. This gives the last statement.  \end{proof}

\begin{lem} Let $q\colon\Toep_{\Hilm}\rightarrow\Toep_{\Hilm}/ J_{\infty}$ be the quotient map. There is a full coaction $\delta\colon\Toep_{\Hilm}/ J_{\infty}\rightarrow\Toep_{\Hilm}/J_{\infty}\otimes\Cst(G)$ satisfying $\delta\circ q=(q\otimes\id_{\Cst(G)})\circ\widetilde{\delta}$. Moreover, the spectral subspace for~$\delta$ at~$g\in G$  is canonically isomorphic to~$\Toep_{\Hilm}^g/\Toep_{\Hilm}^gJ_e.$ 
\begin{proof} Given~$c\in\Toep_{\Hilm}/J_{\infty}$, choose~$b\in\Toep_{\Hilm}$ with~$q(b)=c$ and set $$\delta(q(b))\coloneqq (q\otimes\id_{\Cst(G)})\widetilde{\delta}(b).$$ Lemma~\ref{lem:ideal_description} and Proposition~\ref{prop:induced_coaction} say that this is indeed a well-defined full coaction of~$G$ on~$\Toep_{\Hilm}/J_{\infty}$ . The equality $\delta\circ q=(q\otimes\id_{\Cst(G)})\circ\widetilde{\delta}$ follows from the definition of~$\delta$.

In order to prove the last assertion, let~$(\Toep_{\Hilm}/J_{\infty})^g$ denote the spectral subspace at~$g\in G$ for the coaction~$\delta$. Clearly, the map which sends~$b_g\in \Toep^g_{\Hilm}$ to~$q(b_g)$ vanishes on~$\Toep^g_{\Hilm}J_e$. Moreover, Lemma \ref{lem:ideal_description} implies that this produces an injective map from~$\Toep^g_{\Hilm}/\Toep^g_{\Hilm}J_e$ into~$(\Toep_{\Hilm}/J_{\infty})^g$. That it is also surjective follows by the same argument used in Lemma~\ref{lem:grading_description}.
\end{proof}
\end{lem}

We will often use the above description of the $G$\nb-grading for~$\Toep_{\Hilm}/J_{\infty}$.

\begin{prop}\label{thm:strongly_covariant} The quotient map~$q\colon\Toep_{\Hilm}\rightarrow\Toep_{\Hilm}/J_{\infty}$ is injective on~$\tilde{t}(A)$.
\end{prop}
\begin{proof} We will show that~$\tilde{t}(A)\cap J_e=0$ in~$\Toep_{\Hilm}$. This implies the conclusion by the previous lemma.

Let~$F\subseteq G$ be finite and~$0\neq a\in A$. We claim that~$t_F^e(a)\neq 0$ on~$\Hilm_{F}$. Indeed, if~$a I_{e\vee F}\neq 0$ we are done. Otherwise, $a\not\in I_{K_{\{e,g_1\}}}$ for some~$g_1\in F$, because~$I_{e\vee F}=\bigcap_{\substack{g\in G}}I_{K_{\{e,g\}}}$. Since  $$I_{K_{\{e,g_1\}}}=\bigcap_{\substack{r\in g_1P\cap P}}\ker\varphi_r,$$ there exists~$r_1\in P\cap g_1P$ with~$\varphi_{r_1}(a)\neq 0$. Put $$F_1:=\left\{g\in F\middle|\, K_{\{r_1, g\}}\neq\emptyset\text{ and }r_1\not\in K_{\{r_1, g\}}\right\}.$$ Thus $g_1\not\in F_1$. So~$F_1\subsetneq F$ and
$\Hilm_{r_1}I_{r_1^{-1}(r_1\vee F)}=\Hilm_{r_1}I_{r_1^{-1}(r_1\vee F_1)}.$ Our claim is proved if~$\varphi_{r_1}(a)\neq 0$ on~$\Hilm_{r_1}I_{r_1^{-1}(r_1\vee F_1)}$. This is so, in particular, if~$F_1=\emptyset$ because~$\varphi_{r_1}(a)\neq 0$ on~$\Hilm_{r_1}$. Assume that~$a$ acts trivially on~$\Hilm_{r_1}I_{r_1^{-1}(r_1\vee F_1)}$. Then $$\braket{\varphi_{r_1}(a)\Hilm_{r_1}}{\varphi_{r_1}(a)\Hilm_{r_1}}\cap I_{r_1^{-1}(r_1\vee F_1)}=\{0\}.$$ Thus there exist~$g_2\in F_1$ and~$\xi\in\Hilm_{r_1}$ so that $\braket{\varphi_{r_1}(a)(\xi)}{\varphi_{r_1}(a)(\xi)}\not\in I_{r_1^{-1}(r_1\vee g_2)}$. As a consequence, one can find~$r_2\in K_{\{r_1,g_2\}}$ such that $$\braket{\varphi_{r_1}(a)(\xi)}{\varphi_{r_1}(a)(\xi)}\not\in\ker\varphi_{r_1^{-1}r_2}.$$ We see that~$a\not\in\ker\varphi_{r_2}$ because $\mu_{r_1,r_1^{-1}r_2}\colon\Hilm_{r_1}\otimes_A\Hilm_{r_1^{-1}r_2}\to\Hilm_{r_2}$ is an isomorphism of correspondences. Let $$F_2\coloneqq\left\{g\in F\middle|\; K_{\{r_2, g\}}\neq\emptyset\text{ and }r_2\not\in K_{\{r_2, g\}}\right\}.$$ Notice that~$F_2\subsetneq F_1$ and $$\Hilm_{r_2}I_{r_2^{-1}(r_2\vee F)}=\Hilm_{r_2}I_{r_2^{-1}(r_2\vee F_2)}.$$ If~$\varphi_{r_2}(a)$ vanishes on~$\Hilm_{r_2}I_{r_2^{-1}(r_2\vee F_2)}$, then the same reasoning as above yields~$g_3\in F_2$ and~$r_3\in K_{\{r_2,g_3\}}$ with~$\varphi_{r_3}(a)\neq 0$ on~$\Hilm_{r_3}$. Set $$F_3\coloneqq\left\{g\in F\middle|\; K_{\{r_3, g\}}\neq\emptyset\text{ and }r_3\not\in K_{\{r_3, g\}}\right\}.$$ We then have~$F_3\subsetneq F_2\subsetneq F_1\subsetneq F$ and $\Hilm_{r_3}I_{r_3^{-1}(r_3\vee F)}=\Hilm_{r_3}I_{r_3^{-1}(r_3\vee F_3)}$. This process cannot continue infinitely because~$F$ is finite. So we must stop at some~$r_j$ with~$\varphi_{r_j}(a)\neq0$ on~$\Hilm_{r_j}I_{r_j^{-1}(r_j\vee F)}$. 

Thus~$t_F^e(a)$ is nonzero on~$\Hilm_{F}$. Therefore, for all~$a\in A$, we have that $$\lim_{\substack{F}}\|a\|_F=\lim_{\substack{F}}\|a\|=\|a\|.$$ This completes the proof.\end{proof}

\subsection{The covariance algebra of a product system}

Our goal in this subsection is to associate a $\Cst$\nb-algebra $A\times_{\Hilm}P$ to a given product system~$(\Hilm_p)_{p\in P}$ satisfying two properties: the representation of~$\Hilm$ in~$A\times_{\Hilm}P$ is injective and any representation of~$A\times_{\Hilm}P$ in a $\Cst$\nb-algebra $B$ that is faithful on~$A$ is also faithful on the fixed-point algebra~$(A\times_{\Hilm}P)^{\delta}$ for the canonical gauge coaction of~$G$ on $A\times_{\Hilm}P$, where~$G$ is a group with~$P\subseteq G$. A candidate for~$A\times_{\Hilm}P$ is of course the universal $\Cst$\nb-algebra for strongly covariant representations introduced previously. We shall prove that this is independent of the choice of the group containing~$P$.

\begin{lem}\label{lem:pre_uniqueness_theorem} Let~$P$ be a subsemigroup of a group~$G$ and~$\Hilm$ a product system over~$P$. Let $\psi=\{\psi_p\}_{p\in P}$ be a strongly covariant representation of~$\Hilm$ in a $\Cst$\nb-algebra~$B$. The resulting \Star homomorphism $\widetilde{\psi}\colon\Toep_{\Hilm}\big/J_{\infty}\rightarrow B$ is faithful on~$\Toep^e_{\Hilm}/J_e$ if and only if~$\psi_e$ is injective.
\end{lem}
\begin{proof} If $\widetilde{\psi}$ is faithful on~$\Toep^e_{\Hilm}/J_e$, then Proposition~\ref{thm:strongly_covariant} implies that~$\psi_e$ is injective. Suppose that~$\psi$ is strongly covariant and~$\psi_e$ is faithful. Let us prove that~$\widetilde{\psi}$ is injective on~$\Toep^e_{\Hilm}/J_e$. First, pick $b\in \Toep^e_{\Hilm}$ of the form \begin{equation}\label{eq:generating_set}\tilde{t}(\xi_{p_1})\tilde{t}(\xi_{p_2})^*\ldots\tilde{t}(\xi_{p_{2k-1}})\tilde{t}(\xi_{p_{2k}})^*\end{equation} with $k\in \NN$, $p_1p_2^{-1}\ldots p_{2k-1}p_{2k}^{-1}=e$ and $\xi_{p_i}\in \Hilm_{p_i}$ for all $i\in\{1,2,\ldots,2k\}$. Assume that~$q(b)\neq 0$. This entails~$$K_{\{p_{2i},p_{2i+1}\}}=p_{2i}P\cap p_{2i+1}P\neq\emptyset$$ for each $i\in\{1,2,\ldots,k-1\}$ because, otherwise, $t_{F}\left(\tilde{t}(\xi_{p_{2i}})^*\tilde{t}(\xi_{p_{2i+1}})\right)$ acts trivially on~$\Hilm^+_F$, which would imply~$q(b)=0$. A similar argument employed to \begin{equation}\label{eq:associated_right_ideal}
F\coloneqq\left\{p_{2k}, p_{2k}p_{2k-1}^{-1}p_{2k-2},p_{2k}p_{2k-1}^{-1}p_{2k-2}p_{2k-3}^{-1}p_{2k-4},\ldots,p_{2k}p_{2k-1}^{-1}p_{2k-2}\cdots p_3^{-1}p_{2}\right\}\end{equation} shows that~$K_{F}\neq\emptyset$. This is precisely the right ideal $$p_{2k}p_{2k-1}^{-1}p_{2k-2}\cdots p_3^{-1}p_{2}P.$$ These ideals are used in~\cite{Li:Semigroup_amenability} to study semigroup $\Cst$\nb-algebras. 

We claim that, if~$r\not\in K_{F}$, then
$$\widetilde{\psi}(b)(\psi_r(\Hilm_rI_{r^{-1}(r\vee F)}))=0.$$ Since~$\psi$ is strongly covariant, it suffices to prove that $t_{\{e\}}(b)t^r_{\{e\}}\left(\Hilm_rI_{r^{-1}(r\vee F)}\right)$ vanishes on~$\Hilm_{\{e\}}^+$. First, notice that~$t_{\{e\}}(b)=0$ on~$\Hilm_sI_{s^{-1}(s\vee g)}$ whenever~$s\not\in K_F$. Hence if $K_{\{r,F\}}=\emptyset$, it follows that~$t_{\{e\}}(b)t^r_{\{e\}}(\Hilm_r)=\{0\}$ because the image of~$\Hilm_r$ under~$t^r_{\{e\}}$ sends~$\Hilm_sI_{s^{-1}(s\vee g)}$ to~$\Hilm_{rs}I_{s^{-1}r^{-1}(rs\vee rg)}$ and~$t_{\{e\}}(b)$ vanishes on the latter. We are then left with the case in which $K_{\{r,F\}}\neq\emptyset$. Thus~$r\not\in K_{F}$ implies that~$r\not\in K_{\{g\}}$ for some~$g\in F$. Let $\xi_r\in\Hilm_rI_{r^{-1}(r\vee F)}$. Then $t_{\{e\}}(b)t_{\{e\}}^r(\xi_r)$ vanishes on the direct summand~$\Hilm_sI_{s^{-1}(s\vee g)}$ if $rs\not\in K_{F}$. So assume that~$rs\in K_{F}$. In particular, $rs\in K_{\{r,g\}}.$ Hence $I_{r^{-1}(r\vee F)}\subseteq\ker\varphi_{s}$ and~$t_{\{e\}}^r(\xi_r)=0$ on the direct summand~$\Hilm_sI_{s^{-1}(s\vee g)}$. This concludes the proof that~$t_{\{e\}}(b)t_{\{e\}}^r(\xi_r)=0$ on~$\Hilm_{\{e\}}^+$. Therefore, $\widetilde{\psi}(b)\psi_r(\Hilm_rI_{r^{-1}(r\vee F)})=0$ as claimed.

Now let $b\in\Toep^e_{\Hilm}$ be such that $\widetilde{\psi}(b)=0$. Given $\varepsilon>0$, we must find a finite set $F\subseteq G$ such that $\|b\|_F<\varepsilon$. By Lemma~\ref{lem:grading_description}, there exists 
$b'=\sum_{j=1}^nb_j\in\Toep^e_{\Hilm}$ with $\|b-b'\|<\frac{\varepsilon}{2}$, where each~$b_j$ is of the form~\eqref{eq:generating_set}. For each $j\in\{1,\ldots,n\}$, let $F_j\subseteq G$ be the finite set associated to~$b_j$ as in~\eqref{eq:associated_right_ideal}. Thus $\widetilde{\psi}(b_j)(\psi_r(\Hilm_rI_{r^{-1}(r\vee F_j)}))=0$ if $r\not\in K_{F_j}$. We also set $$F\coloneqq\overset{m}{\underset{j=1}{\bigcup}}F_j$$ and let $\xi=\bigoplus_{r\in P}\xi_r\in\Hilm_F$ with $\|\xi\|_F\leq 1$, where $\xi_r=0$ except for finitely many $r$'s. Then $$\bigg\|\sum\psi_r(\xi_r)^*\widetilde{\psi}(b')^*\widetilde{\psi}(b')\psi_r(\xi_r)\bigg\|=\bigg\|\sum\psi_r(\xi_r)^*\widetilde{\psi}(b-b')^*\widetilde{\psi}(b-b')\psi_r(\xi_r)\bigg\|<\frac{\varepsilon^2}{4}.$$ Since $\psi_e$ is injective and $\widetilde{\psi}(b_j)\psi_r(\xi_r)=0$ if $r\not\in K_{F_j}$, it follows that the left-hand-side above is precisely $\|t_F(b')(\xi)\|^2.$ This implies that~$\|b\|_F<\varepsilon$. Hence~$b$ belongs to~$J_e$ as desired.
\end{proof}

\begin{lem}\label{lem:well_definedness} Let $G$ and $H$  be groups containing~$P$ as a subsemigroup and let~$\Hilm$ be a product system over~$P$. A representation of $\Hilm$ is strongly covariant as in Definition~\textup{\ref{def:strong_covariance}} with respect to~$G$ if and only if it is strongly covariant with respect to~$H$. Thus different groups provide the same notion of strong covariance.
\begin{proof} We may assume that $G=G(P)$ is the universal group of~$P$. By its universal property, there is a group homomorphism $\gamma\colon G\rightarrow H$ extending the embedding of~$P$ into~$H$. Let~$e_G$ and~$e_{H}$ denote the unit elements of~$G$ and $H$, respectively. Let~$\Toep^{e_G}_{\Hilm}$ be the fixed-point algebra for the generalised gauge coaction of~$G$ on~$\Toep_{\Hilm}$. It follows from Lemma~\ref{lem:grading_description} that~$\Toep^{e_G}_{\Hilm}$ is a $\Cst$\nb-subalgebra of~$\Toep^{e_H}_{\Hilm}$, where~$\Toep^{e_H}_{\Hilm}$, in turn, is the fixed-point algebra for the gauge coaction of~$H$ on~$\Toep_{\Hilm}$. Let us prove that~$J_{e_G}\subseteq J_{e_H}$, where~$J_{e_G}\idealin\Toep^{e_G}_{\Hilm}$ and~$J_{e_H}\idealin\Toep^{e_H}_{\Hilm}$ are the ideals described in the construction before Definition~\ref{def:strong_covariance} with respect to the groups~$G$ and~$H$, respectively. 

Indeed, it suffices to show that, given a finite set $F$ with $F\subseteq G$, one has $$I_{p^{-1}(p\vee F)}\supseteq I_{\gamma(p)^{-1}(\gamma(p)\vee \gamma(F))},$$ where $\gamma(F)$ is the range of~$F$ under~$\gamma$. To do so, let $g\in F$. If either $K_{\{p,g\}}=\emptyset$ or $p\in K_{\{p,g\}}$, then $I_{p^{-1}(p\vee g)}=A\supseteq I_{\gamma(p)^{-1}(\gamma(p)\vee \gamma(g))}$. 

Suppose that $K_{\{p,g\}}\neq\emptyset$ and $p\not\in K_{\{p,g\}}$. Given $r\in K_{\{p,g\}}$, $\gamma(r)\in K_{\{\gamma(p),\gamma(g)\}}$ so that $\ker\varphi_{p^{-1}r}\supseteq  I_{\gamma(p)^{-1}(\gamma(p)\vee \gamma(g))}$ because $\gamma$ is a group homomorphism. So we conclude that $I_{p^{-1}(p\vee F)}\supseteq I_{\gamma(p)^{-1}(\gamma(p)\vee \gamma(F))}$ and hence $J_{e_G}\subseteq J_{e_H}$ as asserted. 

Thus we obtain a \Star homomorphism~$\phi\colon\Toep_{\Hilm}/J^G_{\infty}\rightarrow\Toep_{\Hilm}/J^H_{\infty}$. Combining Proposition~\ref{thm:strongly_covariant} with Lemma~\ref{lem:pre_uniqueness_theorem}, we deduce that~$\phi$ is injective on~$\Toep^{e_G}_{\Hilm}/J_{e_G}$. To see that~$\phi$ is an isomorphism, we will show that~$J_{e_H}\subseteq J^G_{\infty}$.

First, let $b\in \Toep^{e_H}_{\Hilm}$ be of the form~\eqref{eq:generating_set}, that is, $$b=\tilde{t}(\xi_{p_1})\tilde{t}(\xi_{p_2})^*\ldots\tilde{t}(\xi_{p_{2k-1}})\tilde{t}(\xi_{p_{2k}})^*,$$ with~$k\in \NN$, $\gamma(p_1)\gamma(p_2)^{-1}\ldots \gamma(p_{2k-1})\gamma(p_{2k})^{-1}=e_H$ and~$\xi_{p_i}\in \Hilm_{p_i}$ for all~$i\in\{1,2,\ldots,2k\}$. We claim that~$p_1p_2^{-1}\ldots p_{2k-1}p_{2k}^{-1}\neq e_G$ in~$G$ entails~$b\in J_{e_H}\cap J^G_{\infty}$. To see this,  we will prove that~$K_F=K_{\gamma(F)}=\emptyset$, where $$F\coloneqq\{p_{2k}, p_{2k}p_{2k-1}^{-1}p_{2k-2},\ldots,p_{2k}p_{2k-1}^{-1}p_{2k-2}\cdots p_3^{-1}p_{2}\}$$ is the finite subset of~$G$ associated to~$b$. 

Let~$r\in K_{\gamma(F)}$.  Then there is a unique~$s_1\in P$ with~$r=p_{2k}s_1$. Here we have omitted~$\gamma$ because it is injective on~$P$. Now~$r$ also lies in $\gamma(p_{2k}p_{2k-1}^{-1}p_{2k-2})P$. So there is a unique $s_2\in P$ so that $r=\gamma(p_{2k}p_{2k-1}^{-1}p_{2k-2})s_2$. This implies that~$\gamma(p_{2k-1}^{-1}p_{2k-2})s_2=s_1$. This is so if and only if~$p_{2k-2}s_2=p_{2k-1}s_1$. Hence~$r=p_{2k}p_{2k-1}^{-1}p_{2k-2}s_2$ in~$G$ as well. Repeating this procedure, we deduce that~$r\in K_F$. Thus~$K_{\gamma(F)}=K_F$, since the inclusion $K_{F}\subseteq K_{\gamma(F)}$ is clear. 

It remains to show that~$K_F=K_{\gamma(F)}=\emptyset$. Let us argue by contradiction and suppose that~$K_F$ is a nonempty subset of~$G$. Hence one can find~$r,s\in P$ with $$p_{2k}p_{2k-1}^{-1}p_{2k-2}\cdots p_3^{-1}p_{2}p_1^{-1}(p_1s)=r.$$ Since~$\gamma$ is injective on~$P$ and $\gamma\left(p_{2k}p_{2k-1}^{-1}p_{2k-2}\cdots p_3^{-1}p_{2}p_1^{-1}\right)=e_H$, it follows that~$p_1s=r$. This gives~$$p_{2k}p_{2k-1}^{-1}p_{2k-2}\cdots p_3^{-1}p_{2}p_1^{-1}r=r$$ and thus~$p_{2k}p_{2k-1}^{-1}p_{2k-2}\cdots p_3^{-1}p_{2}p_1^{-1}=e_G$, contradicting our assumption. Therefore, $K_F=K_{\gamma(F)}=\emptyset$ and hence $b\in J_{e_H}\cap J^G_{\infty}.$ 

As a consequence, the image of~$\Toep^{e_H}_{\Hilm}$ under the quotient map $q\colon\Toep_{\Hilm}\rightarrow\Toep_{\Hilm}/J^G_{\infty}$ lies in the fixed-point algebra~$\Toep^{e_G}_{\Hilm}/J_{e_G}$. Since~$\phi$ is faithful on this latter $\Cst$\nb-algebra and the quotient map~$\Toep_{\Hilm}\rightarrow\Toep_{\Hilm}/J^H_{\infty}$ is precisely the composite~$\phi\circ q$, we conclude that~$J_{e_H}\subseteq J^G_{\infty}$. Therefore $J^G_{\infty}=J^H_{\infty}$. This shows that the notion of covariance described in Definition~\ref{def:strong_covariance} is independent of the choice of the group containing~$P$ as a subsemigroup.
\end{proof}
\end{lem}

The following is our main result:

\begin{thm} Let $P$ be a unital semigroup and let~$\Hilm=(\Hilm_p)_{p\in P}$ be a product system over $P$ of $A$\nb-correspondences. Suppose that~$P$ is embeddable into a group. There is a $\Cst$\nb-algebra~$A\times_{\Hilm}P$ associated to~$\Hilm$ with a representation $j_{\Hilm}\colon\Hilm\rightarrow A\times_{\Hilm}P$ such that the pair $(A\times_{\Hilm}P,j_{\Hilm})$ has the following properties:
\begin{enumerate}
\item[\textup{(C1)}] \label{thm:cuntz_pimsner_algebra} $A\times_{\Hilm}P$ is generated by~$j_{\Hilm}(\Hilm)$ as a $\Cst$\nb-algebra and~$j_{\Hilm}$ is strongly covariant in the sense of Definition~\textup{\ref{def:strong_covariance}}, where the group~$G$ in question may be taken to be any group containing~$P$ as a subsemigroup.

\item[\textup{(C2)}] if $\psi=\{\psi_p\}_{p\in P}$ is a strongly covariant representation of~$\Hilm$ in a $\Cst$\nb-algebra~$B$ with respect to a group containing~$P$, then there is a unique \Star homomorphism $\widehat{\psi}\colon A\times_{\Hilm}P\rightarrow B$ such that $\widehat{\psi}\circ j_p=\psi_p$ for all $p\in P$;

\item[\textup{(C3)}] $j_e$ is faithful and if~$G$ is a group with~$P\subseteq G$ as a subsemigroup, there is a canonical full coaction of~$G$ on~$A\times_{\Hilm}P$ so that a \Star homomorphism $A\times_{\Hilm}P\rightarrow B$ is faithful on the fixed-point algebra~$(A\times_{\Hilm}P)^\delta$ if and only if it is faithful on~$j_e(A)$.
\end{enumerate}

Moreover, up to canonical isomorphism, $(A\times_{\Hilm}P, j_{\Hilm})$ is the unique pair with the properties \textup{(C1)--(C3)}.
\end{thm}

\begin{proof}Let~$G$ be a group containing~$P$ as a subsemigroup. Let~$J_{\infty}$ be the ideal in~$\Toep_{\Hilm}$ as in Lemma~\ref{lem:ideal_description}. That is, $J_{\infty}$ is the ideal generated by $J_e$, which in turn is the ideal in $\Toep^e_{\Hilm}$ constructed before Definition~\ref{def:strong_covariance}. Set~$A\times_{\Hilm}P\coloneqq\Toep_{\Hilm}/J_{\infty}$ and let~$j_{\Hilm}$ be the representation of~$\Hilm$ in $A\times_{\Hilm}P$ given by the composition of $\tilde{t}\colon\Hilm\rightarrow\Toep_{\Hilm}$ with the quotient map $q\colon\Toep_{\Hilm}\rightarrow\Toep_{\Hilm}/J_{\infty}$. By Lemma~\ref{lem:well_definedness}, this does not depend on the chosen group and hence it satisfies (C1). By the universal property of~$\Toep_{\Hilm}$ and again by Lemma~\ref{lem:well_definedness}, $A\times_{\Hilm}P$ also fulfils (C2). Now (C3) follows from Lemma~\ref{lem:pre_uniqueness_theorem}. Uniqueness of~$(A\times_{\Hilm}P, j_{\Hilm})$ is then clear.
\end{proof}

We call~$A\times_{\Hilm}P$ the \emph{covariance algebra} of~$\Hilm$, following the terminology of~\cite{MR2171670} for $\Cst$\nb-algebras associated to partial dynamical systems.

\begin{rem} The proof of Lemma~\ref{lem:well_definedness} also tells us that the fixed-point algebras of the canonical coactions on~$A\times_{\Hilm}P$ of all groups containing~$P$ coincide.
\end{rem}

\begin{example} Let~$G$ be a group and~$(B_g)_{g\in G}$ a saturated Fell bundle over~$G$. View~$(B_g)_{g\in G}$ as a product system over~$G$. For each~$g\in G$ and~$F\subseteq G$ finite, we have that~$I_{g^{-1}(g\vee F)}=B_e$ since~$g\in K_{\{g,h\}}$ for all~$h\in F$. Hence~$J_{e}=\{0\}$ and the associated covariance algebra is isomorphic to the cross sectional~$\Cst$\nb-algebra of~$(B_g)_{g\in G}$.
\end{example}

\section{Relationship to other constructions}\label{sec:examples}

In this section, we relate the covariance algebras of product systems defined here to other constructions in the setting of irreversible dynamical systems. We also give an equivalent notion of strongly covariant representations for compactly aligned product systems over quasi-lattice ordered groups.

\subsection{Relationship to a construction by Sims and Yeend} Let us restrict our attention to semigroups arising from quasi-lattice orders in the sense of~\cite{Nica:Wiener--hopf_operators}: let~$G$ be a group and let~$P$ be a subsemigroup of~$G$ with~$P\cap P^{-1}=\{e\}$. We say that~$(G,P)$ is a \emph{quasi-lattice ordered group} if given elements~$g_1, g_2$ of~$G$ with a common upper bound in~$P$ with respect to the partial order $g_1\leq g_2\Leftrightarrow g_1^{-1}g_2\in P$, then they also have a least upper bound~$g_1\vee g_2$ in~$P$. We write $g_1\vee g_2=\infty$ if~$g_1$ and~$g_2$ have no common upper bound in~$P$. In~\cite{Sims-Yeend:Cstar_product_systems}, Sims and Yeend constructed a $\Cst$\nb-algebra $\mathcal{NO}_{\Hilm}$ from a \emph{compactly aligned} product system~$\Hilm=(\Hilm_p)_{p\in P}$ so that it generalises constructions such as $\Cst$\nb-algebras associated to finitely aligned higher rank graphs and Katsura's Cuntz--Pimsner algebra of a single correspondence. The universal representation of~$\Hilm$ in~$\mathcal{NO}_{\Hilm}$ is quite often faithful, but Example 3.16 of~\cite{Sims-Yeend:Cstar_product_systems} shows that it may fail to be injective even if~$(G,P)$ is totally ordered and~$A$ acts by compact operators on~$\Hilm_p$ for all~$p$ in~$P$. In this subsection, we will see that~$\mathcal{NO}_{\Hilm}$ coincides with~$A\times_{\Hilm}P$ when either the universal representation of~$\Hilm$ in~$\mathcal{NO}_{\Hilm}$ is faithful and $P$ is directed or $\Hilm$ is a faithful product system. In both cases $\mathcal{NO}_{\Hilm}$ satisfies an analogue of (C3)~\cite[Proposition 3.7]{Carlsen-Larsen-Sims-Vittadello:Co-universal}. This subsection is based on~\cite{Carlsen-Larsen-Sims-Vittadello:Co-universal} and~\cite{Sims-Yeend:Cstar_product_systems}.

We first recall the definitions from~\cite{Sims-Yeend:Cstar_product_systems} of Cuntz--Nica--Pimsner covariance and Cuntz--Nica--Pimsner algebra. Let~$\psi=\{\psi_p\}_{p\in P}$ be a representation of~$\Hilm$ in~$B$. For each~$p\in P$, we will denote by $\psi^{(p)}$ the \Star homomorphism from~$\Comp(\Hilm_p)$ to~$B$ obtained as in~\cite{Pimsner:Generalizing_Cuntz-Krieger}. This is defined on a generalised rank\nobreakdash-$1$ operator~$\ket{\xi}\bra{\eta}$ by $$\psi^{(p)}\big(\ket{\xi}\bra{\eta}\big)\coloneqq\psi_p(\xi)\psi_p(\eta)^*.$$

Given a product system $\Hilm=(\Hilm_p)_{p\in P}$, we may use the multiplication maps to define \Star homomorphisms $\iota_p^{pq}\colon\Bound(\Hilm_p)\rightarrow\Bound(\Hilm_{pq})$. Explicitly, $\iota_p^{pq}$ sends $T\in\Bound(\Hilm_p)$ to $\mu_{p,q}\circ T\circ\mu_{p,q}^{-1}$. 

Fix a quasi-lattice ordered group~$(G,P)$. A product system~$\Hilm=(\Hilm_p)_{p\in P}$ is \emph{compactly aligned} if, for all~$p, q\in P $ with $p\vee q<\infty$, we have $$\iota_p^{p\vee q}(T)\iota_q^{p\vee q}(S)\in\Comp(\Hilm_{p\vee q}),\qquad\text{for all } T\in\Comp(\Hilm_p)\text{ and } S\in\Comp(\Hilm_q).$$ 
A representation $\psi=\{\psi_p\}_{p\in P}$ of $\Hilm$ in a $\Cst$\nb-algebra $B$ is \emph{Nica covariant} if for all $p,q\in P$, $T\in \Comp(\Hilm_p)$ and~$S\in\Comp(\Hilm_q)$, we have $$\psi^{(p)}(T)\psi^{(q)}(S)=\begin{cases} \psi^{(p\vee q)}\big(\iota_p^{p\vee q}(T)\iota_q^{p\vee q}(S)\big)  &\text{if } p\vee q<\infty,\\
0 & \text{otherwise.}
\end{cases}$$

The \emph{Nica--Toeplitz algebra} of~$\Hilm$, denoted by $\mathcal{N}\Toep_{\Hilm}$, is the $\Cst$\nb-algebra generated by a copy of~$\Hilm$ that is universal for Nica covariant representations of~$\Hilm$ (see~\cite[Theorem 6.3]{Fowler:Product_systems}).

Fix a compactly aligned product system~$\Hilm=(\Hilm_p)_{p\in P}$. Let $\bar{I}_e\coloneqq A$ and for each~$p\in P\setminus\{e\}$, set $$\bar{I}_p=\bigcap_{\substack{e<s\leq p}}\ker\varphi_{s}\idealin A.$$ Given~$p\in P$, we define a correspondence~$\widetilde{\Hilm}_p\colon A\leadsto A$ by $$\widetilde{\Hilm}_p\coloneqq \bigoplus_{\substack{r\leq p}} \Hilm_r\bar{I}_{r^{-1}p}.$$ For all $s\in P$, there is a \Star homomorphism $\widetilde{\iota}_s^p\colon\Bound(\Hilm_s)\rightarrow\Bound(\widetilde{\Hilm}_p)$ defined by $$\widetilde{\iota}_s^p(T)=\bigg(\bigoplus_{\substack{s\leq r\leq p}}\iota_s^r(T)\vert_{\Hilm_r\bar{I}_{r^{-1}p}}\bigg)\oplus\bigg(\bigoplus_{\substack{s\not\leq r\leq p}}0_{\Hilm_r\bar{I}_{r^{-1}p}}\bigg)\quad\text{for all } T\in\Bound(\Hilm_s).$$

Let $F\subseteq P$ be a finite set and let $T_s\in\Comp(\Hilm_s)$ for each $s\in F$. We say that~$\sum_{\substack{s\in F}}\widetilde{\iota}_s^p(T_s)=0$ for \emph{large}~$p$ if given an arbitrary element~$r$ in $P$, there exists $r'\geq r$ such that $\sum_{\substack{s\in F}}\widetilde{\iota}_s^p(T_s)=0$ for all $p\geq r'$. A representation $\psi$ of $\Hilm$ in a $\Cst$\nb-algebra~$B$ is \emph{Cuntz--Pimsner covariant} according to~\cite[Definition 3.9]{Sims-Yeend:Cstar_product_systems} if $$\sum_{\substack{s\in F}}\psi^{(s)}(T_s)=0$$ whenever~$\sum_{\substack{s\in F}}\widetilde{\iota}_s^p(T_s)=0$ for large~$p$. It is called \emph{Cuntz--Nica--Pimsner covariant} if it is both Nica covariant and Cuntz--Pimsner covariant. 

Suppose that~$\Hilm$ has the extra property that $\widetilde{\iota}_e^p$ is injective on~$A$ for all~$p\in P$. The \emph{Cuntz--Nica--Pimsner algebra} associated to~$\Hilm$, denoted by $\mathcal{NO}_{\Hilm}$, is then the universal $\Cst$\nb-algebra for Cuntz--Nica--Pimsner covariant representations (see \cite[Proposition 3.2]{Sims-Yeend:Cstar_product_systems} for further details). The requirement that $\widetilde{\iota}_e^p$ be faithful for all~$p\in P$ implies that the representation of~$\Hilm$ in~$\mathcal{NO}_{\Hilm}$ is faithful. Sims and Yeend proved in~\cite[Lemma 3.15]{Sims-Yeend:Cstar_product_systems} that this is satisfied whenever~$P$ has the following property: given a nonempty set $F\subseteq P$ that is bounded above, in the sense that there is $p\in P$ with $s\leq p$ for all $s\in F$, then $F$ has a maximal element~$r$. That is, $r\not\leq s$ for all $s\in F\setminus\{r\}$.  

The next example of a product system is given by Sims and Yeend in~\cite[Example 3.16]{Sims-Yeend:Cstar_product_systems}. It consists of a compactly aligned product system for which not all $\widetilde{\iota}_e^p$'s are injective. We recall their example here and describe its associated covariance algebra.

\begin{example}\label{ex:lexicographic} Let $\ZZ\times\ZZ$ be equipped with the lexicographic order and let~$P$ be its positive cone. So $P=\big((\NN\setminus\{0\})\times \ZZ\big)\cup\big(\{0\}\times\NN\big)$ and $e=(0,0)$. Define a product system over $P$ as follows: let $A\coloneqq\CC^2$ and, for each $p\in P$, let $\Hilm_p\coloneqq\CC^2$ be regarded as a Hilbert $A$\nb-module with right action given by coordinatewise multiplication and usual $\CC^2$\nb-valued inner product. Following the notation of~\cite{Sims-Yeend:Cstar_product_systems}, we set $S\coloneqq\{0\}\times\NN$ and for all $p\in S$, we let $\CC^2$ act on $\Hilm_p$ on the left by coordinatewise multiplication, so that $\varphi_p\big((\lambda_1,\lambda_2)\big)\coloneqq(\lambda_1,\lambda_2)\in \Bound(\Hilm_p)$ for all $p\in P$ and $(\lambda_1,\lambda_2)\in A$. For $p\in P\setminus S$, put $\varphi_p\big((\lambda_1,\lambda_2)\big)\coloneqq(\lambda_1,\lambda_1).$ Thus $\ker\varphi_p=\{0\}\times\CC$ for all $p\in P\setminus S$. If $q\in S$, define a correspondence isomorphism $\mu_{p,q}\colon\Hilm_p\otimes_{\CC^2}\Hilm_q\cong\Hilm_{pq}$ by $$(z_1,z_2)\otimes(w_1,w_2)\mapsto(z_1w_1,z_2w_2).$$ For all~$q\in P\setminus S$, define~$\mu_{p,q}\colon\Hilm_p\otimes_{\CC^2}\Hilm_q\cong\Hilm_{pq}$ by $$(z_1,z_2)\otimes(w_1,w_2)\mapsto(z_1w_1,z_1w_2).$$

This is a proper product system~$\Hilm=(\Hilm_p)_{p\in P}$ over $\CC^2$. Thus it is also compactly aligned. Since~$P$ is totally ordered, all representations of~$\Hilm$ are Nica covariant. Sims and Yeend proved that such a product system has no injective Cuntz--Nica--Pimsner covariant representation. Their argument is the following: for all~$p\neq e$, $\bar{I}_p=\ker\varphi_{(0,1)}=\{0\}$. Hence, if~$q\in P\setminus S$, $\widetilde{\iota}_e^q=\varphi_q$ is not injective and any Cuntz--Nica--Pimsner covariant representation of~$\Hilm$ vanishes on~$\ker\varphi_q=\{0\}\times\CC$. 

Let us now describe the associated covariance algebra $A\times_{\Hilm}P$. We will show that $(A\times_{\Hilm}P)^\delta$ is isomorphic to the $\Cst$\nb-algebra of all convergent sequences. To do so, given $p\in P$, write $p=(p_1,p_2)$. We define an isometry $v_p\in\Bound(\ell^2(\NN\times\ZZ))$ by $$v_p(f)(q)=\begin{cases}f(q-p)& \text{if }q_1\geq p_1,\\
0 & \text{otherwise,}\end{cases}$$ where~$f\in\ell^2(\NN\times\ZZ)$ and~$q=(q_1,q_2)\in\NN\times\ZZ$. Thus $v_p^*(f)(q)=f(q+p)$ and~$v_pv_p^*$ is the projection of $\ell^2(\NN\times\ZZ)$ onto the subspace $\ell^2(\NN_{\geq p_1}\times\ZZ)$. In particular, $v_p$ is unitary for all~$p\in S$. 

Let ~$\phi_e\colon\CC^2\rightarrow\Bound(\ell^2(\NN\times\ZZ))$ be the \Star homomorphism given by $$\phi_e\big((\lambda_1,\lambda_2))(f)(q)=\begin{cases}\lambda_2 f(q)& \text{if }q_1=0,\\
\lambda_1f(q) &\text{otherwise.}\end{cases}$$ For all $(z_1,z_2)_p\in \Hilm_p$, put $\phi_p\big((z_1,z_2)_p\big)\coloneqq v_p\phi_e\big((z_1,z_2)).$ This yields a representation $\phi=\{\phi_p\}_{p\in P}$ of $\Hilm$ in $\Bound(\ell^2(\NN\times\ZZ))$. We claim that~$\phi$ is strongly covariant and preserves the topological $\ZZ\times\ZZ$\nb-grading of $A\times_{\Hilm}P$. First, for each finite set~$F\subseteq P$, $$B_F\coloneqq\mathrm{span}\left\{ T_p\middle|\,p\in F, T_p\in\Comp(\Hilm_p)\right\}$$ is a $\Cst$\nb-subalgebra of~$\Toep^e_{\Hilm}$ since~$P$ is totally ordered (see also~\cite[Lemma 3.6]{Carlsen-Larsen-Sims-Vittadello:Co-universal}). In addition, $\Toep^{e}_{\Hilm}=\overline{\bigcup_{\substack{F\subseteq P}}B_F}$. By \cite[Lemma 1.3]{Adji-Laca-Nilsen-Raeburn}, $$J_{e}=\overline{\bigcup_{\substack{F\subseteq P}}J_{e}\cap B_F}.$$ So in order to prove that~$\phi$ is strongly covariant, it suffices to verify that, given a finite set~$F\subseteq P$, one has $$\sum_{\substack{p\in F}}\phi^{(p)}(T_p)=0$$ whenever $\sum_{\substack{p\in F}}j_{\Hilm}^{(p)}(T_p)=0$ in~$A\times_{\Hilm}P$. Here~$T_p\in \Comp(\Hilm_p)$ for each~$p\in F$. So suppose that~$F$ is a finite subset of~$P$ and $\sum_{\substack{p\in F}}j_{\Hilm}^{(p)}(T_p)=0$ in~$A\times_{\Hilm}P$, with~$T_p\in \Comp(\Hilm_p)$. Let~$(\lambda_{1,p},\lambda_{2,p})$ be such that~$T_p=(\lambda_{1,p},\lambda_{2,p})$. Write $$F=\overset{n}{\underset{i=1}{\bigcup}}F_{p_i}$$ with~$p_i<p_{i+1}$ for all $i\in\{1,\ldots,n\}$, where~$F_{p_i}$ is given by all of the elements in~$F$ having first component~$p_i$. Given a finite set $F'\subseteq P$ with $F'\supseteq F$, there is~$r_1=(p_1,q_1)\in P$ such that~$p'<r_1$ for all~$p'\in F'_{p_1}$. Then $$\Hilm_{r_1}I_{r_1^{-1}(r_1\vee F')}=\{0\}\times\CC.$$ So by taking finite sets $F'\subseteq P$ with $F'\supseteq F$, we conclude from the definition of strong covariance that $$\sum_{\substack{p\in F_{p_1}}}\lambda_{2,p}=0.$$  If~$p_2>p_1+1$, we deduce by a similar argument that the sum~$\sum_{\substack{p\in F_{p_1}}}\lambda_{1,p}$ must be zero as well because~$$\iota^r_p(T_p)(\lambda_1,\lambda_2)=\iota_p^r\left((\lambda_{1,p},\lambda_{2,p})\right)(\lambda_{1},\lambda_{2})=(\lambda_{1,p}\lambda_1,\lambda_{1,p}\lambda_2)$$ for all~$r=(r_1,r_2)>p$ with~$r_1>p_1$. In case~$p_2=p_1+1$, then $$\sum_{\substack{p\in F_{p_1}}}\lambda_{1,p}+\sum_{\substack{p\in F_{p_2}}}\lambda_{2,p}=0.$$ Repeating this argument for all of the~$p_i$'s and observing that $$\phi^{(p)}(T_p)(f)(q)=\begin{cases}\lambda_{2,p} f(q)& \text{if }q_1=p_1,\\
\lambda_{1,p}f(q) & \text{if } q_1>p_1,\\
0 &\text{otherwise,}\end{cases}$$ we conclude that~$\phi$ is indeed strongly covariant. The associated representation of~$A\times_{\Hilm}P$ on~$\ell^2(\NN\times\ZZ)$ is faithful on~$(A\times_{\Hilm}P)^\delta$ because it is injective on~$\CC^2$. Its image in~$\Bound(\ell^2(\NN\times\ZZ))$ is the $\Cst$\nb-algebra generated by~$\phi_e(\CC^2)$ and the family of isometries~$\{v_p\vert\, p\in P\}$.

To see that $\widehat{\phi}$ is faithful on $A\times_{\Hilm}P$, consider the canonical unitary representation of the torus~$\TT^2$ on $\ell^2(\NN\times\ZZ)$. Explicitly, the unitary~$U_{z}$ is given by $$U_{z}(f)(q)=z_1^{q_1}z_2^{q_2}f(q), \qquad q=(q_1,q_2)\in\NN\times\ZZ,$$ where $z=(z_1,z_2)\in\TT^2$. This produces a continuous action of~$\TT^2$ on $\widehat{\phi}(A\times_{\Hilm}P)$ by $T\mapsto UTU^*.$ Hence it carries a topological $\ZZ\times\ZZ$\nb-grading. The corresponding spectral subspace at~$(m,n)$ is determined by $$\left\{T\in \widehat{\phi}(A\times_{\Hilm}P)\middle|\,\, U_{z}TU^*_{z}=z_1^{m}z_2^{n}T\text{ for all }z=(z_1,z_2)\in\TT^2\right\}.$$ Since $\ZZ\times\ZZ=P\cup P^{-1}$, it is easy to verify that $\widehat{\phi}$ preserves the grading of~$A\times_{\Hilm}P$. Because~$\ZZ\times\ZZ$ is amenable, Proposition 20.2 and Theorem 20.7 of~\cite{Exel:Partial_dynamical} imply that $\widehat{\phi}$ is an isomorphism onto its image. Its restriction to~$(A\times_{\Hilm}P)^\delta$ yields an isomorphism onto the $\Cst$\nb-algebra of all convergent sequences $$\big\{(\zeta_n)_{n\in\NN}\in\ell^{\infty}(\NN)\vert\,\,\exists\lim_{\substack{n\rightarrow\infty}}\zeta_n\,\big\}.$$ This isomorphism sends $(\lambda_{1,p},\lambda_{2,p})\in\Comp(\Hilm_p)$ to the sequence $(\zeta_n)_{n\in\NN}$ with $$\zeta_n=\begin{cases}\lambda_{2,p} & \text{if }n=p_1,\\
\lambda_{1,p} &\text{if }n>p_1,\\
0 &\text{otherwise.}\end{cases}$$
\end{example}

The task of verifying whether a given representation is strongly covariant or not is considerably simplified when~$\Hilm$ is compactly aligned. The proof of the next proposition is taken from~\cite[Proposition 3.7]{Carlsen-Larsen-Sims-Vittadello:Co-universal} and adapted to our context.

\begin{prop}\label{prop:simplification} Let $\Hilm=(\Hilm_p)_{p\in P}$ be a compactly aligned product system. A representation~$\psi$ of $\Hilm$ in a $\Cst$\nb-algebra~$B$ is strongly covariant if and only if it is Nica covariant and satisfies
\begin{enumerate}
\item[\textup{(C)'}] $\sum_{p\in F}\psi^{(p)}(T_p)=0$ whenever~$\sum_{\substack{p\in F}}t^{(p)}_{F}(T_p)=0$ on~$\Hilm_{F}$, where $F\subseteq P$ is finite and $T_p\in\Comp(\Hilm_p)$ for all $p\in F$.
\end{enumerate}
\begin{proof} Suppose that~$\psi$ is strongly covariant. Let~$p,q\in P$, $T\in\Comp(\Hilm_p)$ and~$S\in \Comp(\Hilm_q)$. If~$p\vee q=\infty$, then~$K_{\{p,q\}}=\emptyset$ so that~$t_F^{(p)}(T)t_F^{(q)}(S)=0$ on the direct summand~$\Hilm_F$ for all finite subsets~$F$ of~$G$. Hence strong covariance implies~$j^{(p)}(T)j^{(q)}(S)=0$ in~$A\times_{\Hilm}P$. Assume that~$p\vee q<\infty$. Let~$T'\in\Comp(\Hilm_{p\vee q})$ be such that~$\iota_p^{p\vee q}(T)\iota_q^{p\vee q}(S)=T'$. Because~$t^{(p)}_F(T)t^{(q)}_F(S)=0$ on~$\Hilm_rI_{r^{-1}(r\vee F)}$ whenever~$r\not\in(p\vee q)P$, it follows that~$t^{(p)}_F(T)t^{(q)}(S)-t_F^{(p\vee q)}(T')=0$ for each finite subset~$F$ of~$G$. Hence $j^{(p)}(T)j^{(q)}(S)=j^{(p\vee q)}(T')$ and~$j_{\Hilm}$ is Nica covariant. 

Now if $\sum_{\substack{p\in F}}t^{(p)}_{F}(T_p)=0$ on~$\Hilm_{F}$ and~$F'\supseteq F$, then $\sum_{\substack{p\in F'}}t^{(p)}_{F'}(T_p)=0$ on~$\Hilm_{F'}$ since~$\Hilm^+_{F'}$ may be viewed as a closed submodule of~$\Hilm_F^+$. So~$\sum_{p\in F}j^{(p)}(T_p)=0$ in~$A\times_{\Hilm}P$ and, in particular, $\sum_{p\in F}\psi^{(p)}(T_p)=0$.

 Conversely, assume that~$\psi$ is Nica covariant and satisfies (C)'. In order to prove that~$\psi$ is strongly covariant, we use the ideas employed in~\cite{Carlsen-Larsen-Sims-Vittadello:Co-universal}. Let~$\mathcal{P}_{\mathrm{fin}}^{\vee}$ denote the set of all finite subsets of~$P$ that are $\vee$\nb-closed. Precisely, $F\in\mathcal{P}_{\mathrm{fin}}^{\vee}$ if it is finite and for all~$p,q\in F$ with~$p\vee q<\infty$, one has $p\vee q\in F$. For each~$F$ in~$\mathcal{P}_{\mathrm{fin}}^{\vee}$, let~$B_F$ denote the subspace of $\mathcal{N}\Toep^e_{\Hilm}$ given by  $$\bigg\{\sum_{\substack{p\in F}}T_p\vert\, T_p\in \Comp(\Hilm_p)\bigg\}.$$ Here we introduce no special notation to identify an element of~$\Comp(\Hilm_p)$ with its image in~$\mathcal{N}\Toep_{\Hilm}.$  We observe that~$B_F$ is a $\Cst$\nb-subalgebra of $\mathcal{N}\Toep^e_{\Hilm}$ and, in addition, $$\mathcal{N}\Toep^e_{\Hilm}=\overline{\bigcup_{\substack{F\in\mathcal{P}_{\mathrm{fin}}^{\vee}}} B_{F}}.$$ Hence, since~$A\times_{\Hilm}P$ is a quotient of~$\mathcal{N}\Toep_{\Hilm}$, Lemma 1.3 of~\cite{Adji-Laca-Nilsen-Raeburn} says that all we must do is prove that $\sum_{p\in F}t_F^{(p)}(T_p)=0$ on~$\Hilm_F$ if~$\sum_{\substack{p\in F}}j^{(p)}(T_p)=0$ in~$A\times_{\Hilm}P$.

Given $r\in P$, it follows from Nica covariance that $j^{(p)}(T_p)j_r(\Hilm_r)=0$ when $p\vee r=\infty$ and $$j^{(p)}(T_p)j_r(\Hilm_r)\subseteq j_r(\Hilm_r)j^{(r^{-1}(r\vee p))}\big(\Comp(\Hilm_{r^{-1}(r\vee p)})\big)$$ otherwise. So $j^{(p)}(T_p)j_r(\Hilm_r\ker\varphi_{r^{-1}(r\vee p)})=0$ if $p\not\leq r$. Therefore, \begin{align*}\underset{p\in F}{\sum}j^{(p)}(T_p)j_r(\Hilm_rI_{r^{-1}(r\vee F)})&=\sum_{\substack{p\in F\\p\leq r}}j^{(p)}(T_p)j_r(\Hilm_rI_{r^{-1}(r\vee F)})\\&=j_r\Big(\sum_{\substack{p\in F\\p\leq r}}t_F^{(p)}(T_p)\Hilm_rI_{r^{-1}(r\vee F)}\Big).\end{align*}Since~$j_{\Hilm}$ is injective, $\underset{p\in F}{\sum}j^{(p)}(T_p)=0$ yields~$\sum_{\substack{p\in F\\p\leq r}}t_F^{(p)}(T_p)\Hilm_rI_{r^{-1}(r\vee F)}=0$ for all~$r\in P$, and we deduce that~$\sum_{\substack{p\in F}}t^{(p)}_{F}(T_p)=0$ on~$\Hilm_{F}$ as desired.
\end{proof} 

\end{prop}

The previous proposition combined with~\cite[Theorem 6.3]{Rennie_groupoidfell} gives us the following:

\begin{cor} Let~$(G,P)$ be a quasi-lattice ordered group and let~$\Hilm=(\Hilm_p)_{p\in P}$ be a compactly aligned product system. Suppose that~$G$ is amenable. If~$A$ is nuclear, then~$A\times_{\Hilm}P$ is nuclear. 
\end{cor}

We denote by~$q_{\mathcal{N}}$ the \Star homomorphism from~$\mathcal{N}\Toep_{\Hilm}$ to~$A\times_{\Hilm}P$ induced by~$j_{\Hilm}=\{j_p\}_{p\in P}$. The proof of the next result is essentially identical to that of Proposition~\ref{prop:simplification}. This is inspired by \cite[Proposition 3.7]{Carlsen-Larsen-Sims-Vittadello:Co-universal}.

\begin{prop}\label{prop:uniqueness_theorem} Let~$\psi$ be an injective Nica covariant representation of~$\Hilm$ in a $\Cst$\nb-algebra~$B$ and let $\psi_{\mathcal{N}}$ denote the induced \Star homomorphism $\mathcal{N}\Toep_{\Hilm}\to B$. Then
$(\ker\psi_{\mathcal{N}})\cap\mathcal{NT}^e_{\Hilm}\subseteq\ker q_{\mathcal{N}}.$
\end{prop}

The following is \cite[Example 3.9]{Carlsen-Larsen-Sims-Vittadello:Co-universal}.

\begin{example}\label{ex:free_group} Let $\FF_2$ denote the free group on two generators $a$ and $b$. Then $\FF_2$ is quasi-lattice ordered and its positive cone $\FF_2^+$ is the unital semigroup generated by~$a$ and~$b$. Define a product system over~$\FF_2^+$ by setting $A\coloneqq\CC$, $\Hilm_a\coloneqq \CC$ and $\Hilm_b\coloneqq\{0\}$, where~$\CC$ is regarded as a Hilbert bimodule over~$\CC$ in the usual way.  So~$\Hilm_{a^n}=\CC$ for all~$n\in\NN$. A subset of~$\FF_2^+$ that is bounded above has a maximal element, so that the representation of~$\Hilm$ in~$\mathcal{NO}_{\Hilm}$ is injective. However, in~\cite{Carlsen-Larsen-Sims-Vittadello:Co-universal} this example illustrates the fact that the conclusion of Proposition~\ref{prop:uniqueness_theorem} may fail for $\mathcal{NO}_{\Hilm}$ if~$P$ is not directed and~$\Hilm$ is non-faithful. 

Define a representation of~$\Hilm$ in $\CC$ by $\psi_p(\lambda_p)=\lambda_p$ for all $p\in P$ and $\lambda_p\in\Hilm_p$. So~$\psi_e$ is faithful. Let $1_a\in\Comp(\Hilm_a)$. Then~$\psi_e(1)-\psi^{(a)}(1_a)=0$ but $\widetilde{\iota}^p_e(1)-\widetilde{\iota}_a^p(1_a)\neq 0$ for all~$p\geq b$. Hence the image of $1-1_a$ in $\mathcal{NO}_{\Hilm}$ is nonzero and it becomes clear that~$\mathcal{NO}_{\Hilm}$ and $A\times_{\Hilm}P$ are not isomorphic, since $j_e(1)-j^{(a)}(1_a)=0$ in the latter. For this example, $A\times_{\Hilm}P$ is the universal $\Cst$-algebra generated by a unitary. That is, $A\times_{\Hilm}P\cong C(\TT)$ with~$j_a(\lambda_a)=\lambda_a z$ and~$j_e(\lambda)=\lambda$, where $z\colon\TT\rightarrow\CC$ is the inclusion function.
\end{example}

\begin{prop}\label{prop:Cuntz-Nica-Pimsner} Let $(G,P)$ be a quasi-lattice ordered group and let~$\Hilm=(\Hilm_p)_{p\in P}$ be a compactly aligned product system over~$P$. Suppose either that~$\Hilm$ is faithful or that~$P$ is directed and the representation of~$\Hilm$ in~$\mathcal{NO}_{\Hilm}$ is injective. Then~$\mathcal{NO}_{\Hilm}$ and $A\times_{\Hilm}P$ are canonically isomorphic to each other.
\end{prop}

\begin{proof} Let~$\bar{j}_{\Hilm}$ denote the representation of~$\Hilm$ in~$\mathcal{NO}_{\Hilm}.$ By Proposition~\ref{prop:uniqueness_theorem}, $\ker\bar{j}_{\mathcal{N}}\cap \mathcal{N}\Toep_{\Hilm}^e\subseteq\ker q_{\mathcal{N}}.$ In particular, $j_{\Hilm}$ is an injective Cuntz--Nica--Pimsner covariant representation of~$\Hilm$ in~$A\times_{\Hilm}P$. Hence~\cite[Proposition 3.7]{Carlsen-Larsen-Sims-Vittadello:Co-universal} implies that the induced \Star homomorphism $j\colon\mathcal{NO}_{\Hilm}\rightarrow A\times_{\Hilm}P$ is faithful on the fixed-point algebra~$\mathcal{NO}^e_{\Hilm}$. Therefore, $\bar{j}_{\mathcal{N}}$ vanishes on~$\ker q_{\mathcal{N}}$ and it factors through~$A\times_{\Hilm}P$. Thus~$\widehat{\bar{j}}_{\mathcal{N}}$ is the inverse of~$j$
\end{proof}

\subsection{Fowler's Cuntz--Pimsner algebra} Recall that a representation~$\psi$ of a correspondence~$\Hilm\colon A\leadsto A$ is \emph{Cuntz--Pimsner covariant} on an ideal~$J\idealin A$ with~$\varphi(J)\subseteq\Comp(\Hilm)$ if $$\psi(a)=\psi^{(1)}(\varphi(a))$$ for all~$a\in J$ \cite{Muhly-Solel:Tensor}. Fowler defined the Cuntz--Pimsner algebra of~$\Hilm=(\Hilm_p)_{p\in P}$ to be the universal $\Cst$\nb-algebra for representations that are Cuntz--Pimsner covariant on~$J_p\coloneqq \varphi^{-1}(\Comp(\Hilm_p))$ for all~$p\in P$ \cite[Proposition 2.9]{Fowler:Product_systems}. We denote Fowler's Cuntz--Pimsner algebra by~$\CP_{\Hilm}$. 

Our next result provides sufficient conditions for $A\times_{\Hilm}P$ to coincide with Fowler's Cuntz--Pimsner algebra if~$P$ is a cancellative right Ore monoid, that is,  $P$ is cancellative and~$pP\cap qP\neq\emptyset$ for all~$p,q\in P$. In this case, $P$ can be embedded in a group~$G$ so that~$G=PP^{-1}$.

\begin{prop}\label{prop:fowler_cuntz_pimsner_algebra} Let $P$ be a cancellative right Ore monoid and let~$\Hilm=(\Hilm_p)_{p\in P}$ be a product system that is faithful and proper. Then $A\times_{\Hilm}P$ is canonically isomorphic to Fowler's Cuntz--Pimsner algebra.
\begin{proof} Observe that~$J_p=A$ for all $p\in P$. We begin by verifying that the representation of~$\Hilm$ in~$A\times_{\Hilm}P$ is Cuntz--Pimsner covariant on~$A$ for each~$p$ in~$P$.  Indeed, set~$F\coloneqq\{p\}$. Since~$\Hilm$ is faithful, it follows that $I_{r^{-1}(r\vee p)}=\{0\}$ if $r\not\in pP$. Thus $$\Hilm_{\{p\}}=\bigoplus_{\substack{r\in pP}}\Hilm_r.$$ Hence~$j_{\Hilm}$ is Cuntz--Pimsner covariant on~$J_p$ for each~$p\in P$. We then obtain a \Star homomorphism $j\colon\CP_{\Hilm}\rightarrow A\times_{\Hilm}P$. 

By \cite[Theorem 3.16]{Albandik-Meyer:Product}, we may view~$\CP^e_{\Hilm}$ as the inductive limit of~$\big(\Comp(\Hilm_p)\big)_{p\in P}$. Thus~$j$ is faithful on $\CP^e_{\Hilm}$ since it is so on all of the $\Comp(\Hilm_p)$'s. The quotient map $q\colon\Toep_{\Hilm}\rightarrow A\times_{\Hilm}P$ is the composition of~$j$ with the quotient map from~$\Toep_{\Hilm}$ to~$\CP_{\Hilm}$. Hence the representation of $\Hilm$ in this latter $\Cst$\nb-algebra must vanish on $J_{\infty}$. The induced \Star homomorphism~$A\times_{\Hilm}P\rightarrow\CP_{\Hilm}$ is then the inverse of~$j$.
\end{proof}
\end{prop}

We will say that a correspondence~$\Hilm\colon A\leadsto A$ is a \emph{Hilbert $A$\nb-bimodule} if $\Hilm$ carries a structure of \emph{left} Hilbert $A$\nb-module such that $\varphi(\BRAKET{\xi}{\eta})\zeta = \xi\braket{\eta}{\zeta}$
  for all \(\xi, \eta, \zeta\in \Hilm\), where $\BRAKET{\cdot}{\cdot}$ denotes the left $A$\nb-valued inner product.  Let~$\Hilm=(\Hilm_p)_{p\in P}$ be a product system of Hilbert bimodules. For each~$p\in P$, set~$J_{\Hilm_p}\coloneqq\BRAKET{\Hilm_p}{\Hilm_p}$. So~$\Hilm[J]_{\Hilm}=\{J_{\Hilm_p}\}_{p\in P}$ is a family of ideals in~$A$ satisfying $J_{\Hilm_p}=\varphi_p^{-1}(\Comp(\Hilm_p))\cap(\ker\varphi_p)^\perp$ for all~$p\in P$. Let~$\CP_{\Hilm[J]_{\Hilm},\Hilm}$ denote the universal $\Cst$\nb-algebra for representations of~$\Hilm$ that are Cuntz--Pimsner covariant on~$J_{\Hilm_p}$ for all~$p\in P$. 

\begin{prop} Suppose that $\Hilm=(\Hilm_p)_{p\in P}$ is a product system of Hilbert bimodules with the following properties: \begin{enumerate}
\item[\textup{(i)}] $\BRAKET{\Hilm_p}{\Hilm_p}\BRAKET{\Hilm_r}{\Hilm_r}\subseteq\BRAKET{\Hilm_{q}}{\Hilm_{q}}$ for some $q\in pP\cap rP$ whenever $pP\cap rP\neq\emptyset$;
\item[\textup{(ii)}] $\BRAKET{\Hilm_p}{\Hilm_p}\BRAKET{\Hilm_r}{\Hilm_r}=\{0\}$ if $pP\cap rP=\emptyset$;
\end{enumerate} Then $A\times_{\Hilm}P$ is canonically isomorphic to~$\CP_{\Hilm[J]_{\Hilm},\Hilm}$.
\end{prop}
\begin{proof} We begin by proving that the canonical representation of~$\Hilm$ in~$A\times_{\Hilm}P$ factors through~$\CP_{\Hilm[J]_{\Hilm},\Hilm}.$ Let~$p\in P$ and let~$r\in P$ be such that~$r\not\in pP$. Condition (ii) in the statement entails~$\varphi_r(\BRAKET{\Hilm_p}{\Hilm_p})\Hilm_r=\{0\}$ if $pP\cap rP=\emptyset$. Suppose that~$pP\cap rP\neq\emptyset$. Then there exists~$q\in pP\cap rP$ such that $$\BRAKET{\Hilm_p}{\Hilm_p}\BRAKET{\Hilm_r}{\Hilm_r}\subseteq\BRAKET{\Hilm_{q}}{\Hilm_{q}}.$$ In particular, if~$F\subseteq G$ is a finite subset and~$p\in F$, it follows that $$\BRAKET{\Hilm_p}{\Hilm_p}\BRAKET{\Hilm_r}{\Hilm_rI_{r^{-1}(r\vee F)}}\subseteq \BRAKET{\Hilm_p}{\Hilm_p}\BRAKET{\Hilm_r}{\Hilm_r\ker\varphi_{r^{-1}q}}=\{0\}$$ because $\BRAKET{\Hilm_{q}}{\Hilm_{q}}$ acts faithfully on~$\Hilm_{q}$. Hence $\varphi_r(\BRAKET{\Hilm_p}{\Hilm_p})\Hilm_rI_{r^{-1}(r\vee F)}=\{0\}$. So given~$a$ in~$\BRAKET{\Hilm_p}{\Hilm_p}$, $t_S^e(a)-t_S^{(p)}(\varphi_p(a))=0$ on~$\Hilm_F$ for all finite subsets~$F$ of~$G$ with~$F\supseteq \{p\}$. This shows that~$j=\{j_p\}_{p\in P}$ is an injective representation of~$\Hilm$ in~$A\times_{\Hilm}P$ that is Cuntz--Pimsner covariant on~$J_{\Hilm_p}=\BRAKET{\Hilm_p}{\Hilm_p}$ for all~$p\in P$. So it induces a \Star homomorphism $\phi\colon \CP_{\Hilm[J]_{\Hilm},\Hilm}\to A\times_{\Hilm}P$. As a consequence, the universal representation of~$\Hilm$ in $\CP_{\Hilm[J]_{\Hilm},\Hilm}$ is injective. Combining this with conditions~(i) and (ii) in the statement, we conclude that the fixed-point algebra~$\CP^e_{\Hilm[J]_{\Hilm},\Hilm}$ for the canonical coaction of~$G$ on~$\CP_{\Hilm[J]_{\Hilm},\Hilm}$ is isomorphic to~$A$. Since~$\phi$ is a surjective grading-preserving \Star homomorphism, the universal property of~$A\times_{\Hilm}P$ tells us that~$\phi$ is an isomorphism.
\end{proof}

\subsection{Semigroup \texorpdfstring{$\Cst$\nb-}{C*-}algebras} The semigroup $\Cst$\nb-algebra as introduced by Murphy in \cite{Murphy:Crossed_semigroups} is the universal $\Cst$\nb-algebra for representations of~$P$ by isometries, also called \emph{isometric} representations. Unlike the group case, the resulting $\Cst$\nb-algebra is usually badly behaved. For instance, it is not nuclear even when the semigroup in question is~$\NN\times\NN$ (see~\cite{murphy1996}). For semigroups that are positive cones of quasi-lattice ordered groups, Nica considered in~\cite{Nica:Wiener--hopf_operators} a sub-class of isometric representations, namely, those satisfying the Nica covariance condition. In this setting, he also introduced a notion of amenability for a quasi-lattice ordered group $(G,P)$ and proved, for instance, that~$(\FF_n,\FF_n^+)$ is amenable. Xin Li realised that one should also take into account the family of right ideals of~$P$ and proposed a construction generalising that of Nica to left cancellative semigroups~\cite{Li:Semigroup_amenability}. In analogy with the group case, he was able to relate amenability of a semigroup to its associated $\Cst$\nb-algebra (see \cite[Section 4]{Li:Semigroup_amenability}). In this subsection, we study the relationship between covariance algebras and the semigroup $\Cst$\nb-algebras of Xin Li. Under a certain assumption involving the family of constructible right ideals of~$P$, we will show that we can recover the semigroup $\Cst$\nb-algebra of Xin Li from the covariance algebra of a certain product system. This is obtained in~\cite[Section 5]{Albandik-Meyer:Product} for Ore monoids. 

Let us first recall Li's construction. Assume that~$G$ is generated by~$P$. Given $\alpha=(p_1,p_2,\ldots,p_{2k})\subseteq P$, define  \begin{equation}\label{eq:finite_set} F_{\alpha}=\{p_{2k}^{-1}p_{2k-1}, p_{2k}^{-1}p_{2k-1}p_{2k-2}^{-1}p_{2k-3},\ldots, p_{2k}^{-1}p_{2k-1}p_{2k-2}^{-1}\cdots p_2^{-1}p_1\}.\end{equation}  Then $K_{\{F_{\alpha}, e\}}$ is a right ideal in~$P$. This corresponds to the right ideal $$ p_{2k}^{-1}p_{2k-1}p_{2k-2}^{-1}\cdots p_2^{-1}p_1P$$ in the notation of~\cite{Li:Semigroup_amenability}. Given words $\alpha_1, \alpha_2,\ldots,\alpha_n$ in $P$, the intersection $$\overset{n}{\bigcap_{\substack{i=1}}}K_{\{F_{\alpha_i},e\}}$$ is again a right ideal in~$P$. Let~$\Hilm[J]$ be the smallest family of right ideals of~$P$ containing the "constructible" right ideals as above and the empty set~$\emptyset$. This is closed under finite intersection. In addition, if~$S\in\mathcal{J}$, then~$pS\in\mathcal{J}$  and~$p^{-1}S\in\mathcal{J}$, where~$pS$ and~$p^{-1}S$ denote the image and pre-image of~$S$, respectively, under left multiplication by~$p$. The following is \cite[Definition 3.2]{Li:Semigroup_amenability}.

\begin{defn} Let $P$ be a subsemigroup of a group $G$. The \emph{semigroup $\Cst$\nb-algebra} of~$P$, denoted by~$\Cst_s(P)$, is the universal $\Cst$\nb-algebra generated by a family of isometries~$\{v_p\vert\, p\in P\}$ and projections~$\{e_S\vert\, S\in\Hilm[J]\}$ satisfying the following:
\begin{enumerate}\label{def:semigroup_Cstaralgebra}
\item[\textup{(i)}] $v_pv_q=v_{pq}$
\item[\textup{(ii)}] $e_{\emptyset}=0$
\item[\textup{(iii)}] $v_{p_1}^*v_{p_2}\cdots v_{p_{2k-2}} v_{p_{2k-1}}^*v_{p_{2k}}=e_{K_{\{F_{\alpha},e\}}}$ whenever~$\alpha=(p_1,p_2,\ldots,p_{2k})$ is a word in~$P$ with~$p_1^{-1}p_2\cdots p_{2k-2}p_{2k-1}^{-1}p_{2k}=e$ in~$G$.
\end{enumerate}
\end{defn}

 The family~$\Hilm[J]$ of right ideals of~$P$ is called \emph{independent} (see \cite[Definition 2.26]{Li:Semigroup_amenability}) if given a right ideal of~$P$ of the form $$S=\underset{i=1}{\overset{m}{\bigcup}}S_i,$$ with $ S_i\in\mathcal{J}$ for all $i\in\{1,\ldots,m\}$, then~$S=S_i$ for some $i\in\{1,\ldots,m\}.$ By \cite[Lemma 3.3]{Li:Semigroup_amenability}, $e_{S_1}e_{S_2}=e_{S_1\cap S_2}$ in~$\Cst_s(P)$ for all $S_1, S_2$ in $\Hilm[J]$ and hence the closed linear span of the projections $\{e_{S}\vert\,S\in\Hilm[J]\}$ is a commutative $\Cst$\nb-subalgebra of~$\Cst_s(P)$. If~$\Hilm[J]$ is independent, this $\Cst$\nb-subalgebra is canonically isomorphic to the $\Cst$\nb-subalgebra of~$\ell^{\infty}(P)$ generated by the characteristic functions on elements of~$\Hilm[J]$~\cite[Corollary 3.4]{Li:Semigroup_amenability}. Let us denote this latter $\Cst$\nb-algebra by~$A$. That is, $$A=\overline{\mathrm{span}}\{\chi_{S}\vert S\in\Hilm[J]\},$$ where~$\chi_S\in\ell^{\infty}(P)$ is the characteristic function on~$S$. This will be the coefficient algebra of our product system~$\Hilm$. The idea is taken from~\cite[Section 5]{Albandik-Meyer:Product}. Our assumption, however, is different: we require~$P$ to be embeddable in a group, as usual. So we follow~\cite[Definition 3.2]{Li:Semigroup_amenability}.   
 
There is a semigroup action~$\beta\colon P\rightarrow\text{End}(A)$ by injective endomorphisms with hereditary range as follows. Let~$\beta_p$ be defined by $\chi_{S}\mapsto\chi_{pS}$. Its range~$\beta_p(A)$ is the corner~$\chi_{pP}A\chi_{pP}$ and hence it is hereditary. This gives us a product system over~$P$ as in~\cite{Larsen:Crossed_abelian}. The correspondence $\Hilm_p\colon A\leadsto A$ is $A\chi_{pP}$ with the following structure: we use the inverse~$\beta_p^{-1}$ to define the~$A$\nb-valued inner product, so that $$\braket{a\chi_{pP}}{b\chi_{pP}}\coloneqq \beta_p^{-1}(\chi_{pP}a^*\cdot b\chi_{pP}).$$ In particular, $\braket{\chi_S\chi_{pP}}{\chi_{pP}}=\chi_{(p^{-1}S)\cap P}$ for all~$S\in \Hilm[J]$. The right action of~$A$ on~$\Hilm$ is implemented by~$\beta_p$. That is, $(b\chi_{pP})\cdot\chi_{S}=b\chi_{pS}.$ The left action is then defined by left multiplication $a\cdot(b\chi_{pP})=ab\chi_{pP}$. Finally, the isomorphism~$\mu_{p,q}\colon\Hilm_p\otimes_A\Hilm_q\cong\Hilm_{pq}$ sends~$a\chi_{pP}\otimes_Ab\chi_{qP}$ to $a\chi_{pP}\beta_p(b)\chi_{pqP}$.

\begin{prop}\label{prop:independent_right_ideals} Suppose that~$\Hilm[J]$ is independent. The semigroup $\Cst$\nb-algebra  $\Cst_s(P)$ is naturally isomorphic to~$A\times_{\Hilm}P$.
\begin{proof} Let us define a \Star homomorphism~$\Cst_s(P)\rightarrow A\times_{\Hilm}P$ by using the universal property of~$\Cst_s(P)$. For each~$p\in P$, put $u_p\coloneqq j_p(\chi_{pP})$. Thus~$u\colon p\mapsto u_p$ is an isometric representation of~$P$ in~$A\times_{\Hilm}P.$ Given~$S\in\Hilm[J]$, set~$\bar{e}_{S}=\chi_S$. In order to prove that this data also satisfies condition (iii) of Definition \ref{def:semigroup_Cstaralgebra}, let $\alpha=(p_1,p_2,\ldots,p_{2k})$ be a word in~$P$ with~$p_1^{-1}p_2\cdots p_{2k-2}p_{2k-1}^{-1}p_{2k}=e$ . Let $F_{\alpha}$ be as in \eqref{eq:finite_set}. Let us show that \begin{equation}\label{eq:semigroup_representation}
t_{F_{\alpha}}\big(\widetilde{t}_{p_1}(\chi_{p_1P})^*\widetilde{t}_{p_2}(\chi_{p_2P})\cdots \widetilde{t}_{p_{2k-1}}(\chi_{p_{2k-1}P})^*\widetilde{t}_{p_{2k}}(\chi_{p_{2k-1}P})-\widetilde{t}(\chi_{K_{\{F_{\alpha},e\}}})\big)=0
\end{equation} on $\Hilm_{F_{\alpha}}$. This is clearly true if $K_{\{F_{\alpha},e\}}=\emptyset$ or~$K_{\{F_{\alpha},e\}}=P$. So let us assume otherwise. The ideal~$I_{e\vee F_{\alpha}}\idealin A$ is generated by the characteristic functions on the right ideals that have empty intersection with $$K_{\{F_{\alpha},e\}}= p_{2k}^{-1}p_{2k-1}p_{2k-2}^{-1}\cdots p_2^{-1}p_1P$$ so that $\chi_{K_{\{F_{\alpha},e\}}}I_{e\vee F_{\alpha}}=0$.  Similarly, let $r\not\in K_{\{F_{\alpha},e\}}$. Observe that $\chi_{K_{\{F_{\alpha},e\}}}$ vanishes on~$\Hilm_r$ whenever~$rP\cap K_{\{F_{\alpha},e\}}=\emptyset$. If $rP\cap K_{\{F_{\alpha},e\}}\neq\emptyset$, then $I_{r^{-1}(r\vee F_{\alpha})}$ consists of those functions in~$A$ that vanish on $P\cap r^{-1}K_{\{F_{\alpha},e\}}$. In particular, \begin{align*}
\varphi_r(\chi_{K_{\{F_{\alpha},e\}}})(\chi_{rP})\cdot I_{r^{-1}(r\vee F_{\alpha})}&=\chi_{K_{\{F_{\alpha},e\}}\cap rP}\beta_r(I_{r^{-1}(r\vee F_{\alpha})})\\
&=\chi_{K_{\{F_{\alpha},e\}}\cap rP}I_{(r\vee F_{\alpha})}=\{0\}.
\end{align*}

For~$r\in K_{\{F_{\alpha},e\}}$, one may easily verify that the left-hand side of~\eqref{eq:semigroup_representation} also vanishes on~$\Hilm_r$. This proves our claim that condition (iii) of Definition~\ref{def:semigroup_Cstaralgebra} is satisfied. So we obtain a \Star homomorphism $\phi\colon\Cst_s(P)\rightarrow A\times_{\Hilm}P.$ This sends~$v_p$ to the isometry~$u_p$ and~$e_{S}$ to~$j_e(\chi_S)$. 

In order to define a representation of~$\Hilm$ in $\Cst_s(P)$, we invoke the assumption that~$\Hilm[J]$ is independent. As mentioned before the statement, in this case the commutative~$\Cst$\nb-subalgebra of~$\Cst_s(P)$ generated by the projections $\{e_{S}\vert\,S\in\Hilm[J]\}$ is canonically isomorphic to~$A$. So there is a \Star homomorphism $A\rightarrow\Cst_s(P)$ which maps $\chi_{S}$ to~$e_S$. Lemmas 2.8 and 3.3 of~\cite{Li:Semigroup_amenability} imply the relations $$v_pe_Sv_p^*=e_{pS},\qquad v_p^*e_Sv_p=e_{p^{-1}S\cap P}$$ in~$\Cst_s(P)$ for all~$p\in P$ and~$S\in\Hilm[J]$. Hence the map which sends~$\chi_{pP}\in\Hilm_p$ to the isometry~$v_p$ together with the \Star homomorphism~$\chi_{S}\mapsto e_S$ gives us a representation of~$\Hilm$ in~$\Cst_s(P)$. The induced \Star homomorphism $\Toep_{\Hilm}\rightarrow\Cst_s(P)$ preserves the $G$\nb-grading for the coaction of~$G$. Moreover, it follows from condition (iii) and the equality~$v_e=1$ that the fixed-point algebra~$\Cst_s(P)^e$ for such a coaction is the $\Cst$\nb-algebra generated by the projections $\{e_S\vert\, S\in\Hilm[J]\}$, which in turn is isomorphic to~$A$. Hence~$\phi$ is injective on~$\Cst_s(P)^e.$ By the same argument employed in the proof of Proposition~\ref{prop:fowler_cuntz_pimsner_algebra}, we conclude that~$\phi$ is an isomorphism.
\end{proof}

\end{prop}

The proof of the previous proposition shows that, in general, $A\times_{\Hilm}P$ is a quotient of~$\Cst_s(P)$. It is isomorphic to the $\Cst$\nb-algebra ${\Cst}_s^{(\cup)}(P)$ in the notation of~\cite{Li:Semigroup_amenability}. Indeed, let~$$\Hilm[J]^{\cup}\coloneqq\bigg\{\underset{i=1}{\overset{m}{\bigcup}}S_i\vert S_i\in\mathcal{J}\bigg\}.$$ Let ${\Cst}^{(\cup)}_s(P)$ be the universal $\Cst$\nb-algebra generated by isometries~$\{v_p\vert\, p\in P\}$ and projections~$\{e_S\vert\, S\in\Hilm[J]^{\cup}\}$ satisfying the conditions (i)--(iii) of Definition \ref{def:semigroup_Cstaralgebra} with the additional relation
\begin{enumerate}
\item[\textup{(iv)}] $e_{S_1\cup S_2}=e_{S_1}+e_{S_2}-e_{S_1\cap S_2}$ for all $S_1, S_2\in\Hilm[J]^{\cup}$.
\end{enumerate}

The $\Cst$\nb-algebra ${\Cst}^{(\cup)}_s(P)$ coincides with~${\Cst}_s(P)$ whenever~$\Hilm[J]$ is independent (see \cite[Proposition 2.24]{Li:Semigroup_amenability}). The next result generalises Proposition~\ref{prop:independent_right_ideals}.

\begin{cor} The semigroup $\Cst$\nb-algebra ${\Cst}^{(\cup)}_s(P)$ is naturally isomorphic to~$A\times_{\Hilm}P.$
\begin{proof} It follows from~\cite[Lemma 3.3]{Li:Semigroup_amenability} and \cite[Corollary 2.22]{Li:Semigroup_amenability} that the $\Cst$\nb-subalgebra of ${\Cst}^{(\cup)}_s(P)$ generated by the $e_S$'s is naturally isomorphic to~$A$. Again condition (iii) of Definition \ref{def:semigroup_Cstaralgebra} implies that this $\Cst$\nb-subalgebra coincides with the fixed-point algebra for the canonical coaction of~$G$ on~${\Cst}^{(\cup)}_s(P)$. Now we may employ the same argument used in the proof of Proposition \ref{prop:independent_right_ideals} to obtain an isomorphism $\Cst_s(P)\cong A\times_{\Hilm}P.$
\end{proof}

\end{cor}

\subsection{Crossed products by interaction groups} In this subsection, we will show how Exel's crossed products by interaction groups fit into our approach. This notion of crossed products was introduced in \cite{Exel:New_look} in order to study semigroups of unital and injective endomorphisms. We first recall some concepts from his work, although many details will be omitted. An \emph{interaction group} is a triple $(A,G,V)$, where~$A$ is a unital $\Cst$\nb-algebra, $G$ is a group and~$V$ is a \emph{partial representation} of~$G$ in the Banach algebra of bounded operators on~$A$. This consists of a family~$\{V_g\}_{g\in G}$ of continuous operators on~$A$ with $V_1=\id_A$ and~$$V_gV_hV_{h^{-1}}=V_{gh}V_{h^{-1}},\quad V_{g^{-1}}V_gV_{h}=V_{g^{-1}}V_{gh}$$ for all~$g,h\in G.$ It follows that~$E_g\coloneqq V_gV_{g^{-1}}$ is an idempotent for each~$g\in G$ and $E_gE_h=E_gE_h$, $g,h\in G$. The partial representation is also assumed to satisfy the following conditions:
\begin{enumerate}
\item[\textup{(i)}] $V_g$ is a positive map,
\item[\textup{(ii)}] $V_g(1)=1$,
\item[\textup{(iii)}] $V_g(ab)=V_g(a)V_g(b)$ if~$a$ or~$b$ belong to the range of~$V_{g^{-1}}$.
\end{enumerate}

For all~$g\in G$, the idempotent~$E_g$ is a conditional expectation onto the range of~$V_g$. An interaction group is said to be \emph{nondegenerate} if~$E_g$ is faithful for all~$g$ in~$G$. That is,  $E_g(a^*a)=0$ implies $a=0$ (see \cite[Definition 3.3]{Exel:New_look}). 

Frow now on let us fix a nondegenerate interaction group~$(A,G,V)$. Given a unital $\Cst$\nb-algebra~$B$, recall that $v\colon G\rightarrow B$ is a~\emph{\Star partial representation} if it is a partial representation satisfying~$v_g^*=v_{g^{-1}}$ for all $g\in G$. A \emph{covariant representation} of $(A,G,V)$ in~$B$ is a pair $(\pi,v)$, where~$\pi\colon A\rightarrow B$ is a unital \Star homomorphism and~$v$ is a \Star partial representation of~$G$ in~$B$ such that $$v_g\pi(a)v_{g^{-1}}=\pi(V_g(a))v_gv_{g^{-1}}.$$ The \emph{Toeplitz algebra} of $(A,G,V)$, denoted by~$\Toep(A,G,V)$, is the universal $\Cst$\nb-algebra for covariant representations of~$(A,G,V)$. It is generated by a copy of~$A$ and elements~$\{\widehat{s}_g\}_{g\in G}$ so that 
$\widehat{s}\colon g\mapsto\widehat{s}_g$ is a \Star partial representation and the pair~$(j_V,\widehat{s})$ is a covariant representation of~$(A,G,V)$ in~$\Toep(A,G,V)$, where $j_V\colon A\rightarrow\Toep(A,G,V)$ denotes the canonical embedding. 

In order to recall the notion of redundancy introduced by Exel in~\cite{Exel:New_look}, let us first define certain subspaces of~$\Toep(A,G,V)$. Given a word $\alpha=(g_1,g_2,\ldots,g_n)$ in~$G$, set $$\widehat{s}_\alpha=\widehat{s}_{g_1}\widehat{s}_{g_2}\cdots \widehat{s}_{g_n}.$$ Let $\widehat{\Hilm[M]}_{\alpha}=j_V(A)\widehat{s}_{\alpha}j_V(A)$ and $e_{\alpha}\coloneqq \widehat{s}_{\alpha}\widehat{s}_{\alpha^{-1}}$, where $\alpha^{-1}=(g^{-1}_n,\cdots,g_2^{-1},g_1^{-1})$. Then $\widehat{s}_{\alpha}j_V(a)\widehat{s}_{\alpha^{-1}}=j_V(V_{\alpha}(a))e_{\alpha}$ and by \cite[Proposition 2.7]{Exel:New_look},~$e_{\alpha}$ is also an idempotent. The subspace $\widehat{\Hilm[Z]}_{\alpha}$ associated to the word~$\alpha$ will be the closed linear span of elements of the form $$j_V(a_0)\widehat{s}_{g_1}j_V(a_1)\widehat{s}_{g_2}\cdots\widehat{s}_{g_n}j_V(a_n)$$ with $a_0,a_1, a_2,\ldots,a_n\in A$. We set $\widehat{\Hilm[Z]}_{\alpha}=j_V(A)$ in case $\alpha$ is the empty word. Observe that we always have $\widehat{\Hilm[M]}_{\alpha}\subseteq\widehat{\Hilm[Z]}_{\alpha}$. We also associate a finite subset of~$G$ to the word~$\alpha$ by letting $$\mu(\alpha)=\{e,g_1,g_1g_2,\ldots,g_1g_2\cdots g_n\},$$ so that $\mu(\alpha)=\{e\}$ if~$\alpha$ is the empty word. We further let~$\overset{.}{\alpha}=g_1g_2\cdots g_n.$ If~$\overset{.}{\alpha}=e$, it follows that $\mu(\alpha)=\mu(\alpha^{-1})$. We denote by $\Hilm[W]_{\alpha}$ the set of all words~$\beta$ in~$G$ with $\mu(\beta)\subseteq\mu(\alpha)$ and $\overset{.}{\beta}=e$ and let $$\widehat{\Hilm[Z]}^{\mu(\alpha)}\coloneqq\overline{\mathrm{span}}\{c_{\beta}\vert\, c_{\beta}\in\widehat{\Hilm[Z]}_{\beta},\,\beta\in\Hilm[W]_{\alpha}\}.$$ This is a $\Cst$\nb-subalgebra of $\Toep(A,G,V)$ since~$\beta\in\Hilm[W]_{\alpha}$ if and only if~$\beta^{-1}\in\Hilm[W]_{\alpha}$ and~$\Hilm[W]_{\alpha}$ is also closed under concatenation of words (see \cite[Proposition 4.7]{Exel:New_look} for further details). In addition, $\widehat{\Hilm[Z]}^{\mu(\alpha)}\widehat{\Hilm[M]}_{\alpha}\subseteq\widehat{\Hilm[M]}_{\alpha}.$

\begin{defn}Let $\alpha$ be a word in~$G$. We say that $c\in\widehat{\Hilm[Z]}^{\mu(\alpha)}$ is an \emph{$\alpha$\nb-redundancy} if $c\widehat{\Hilm[M]}_{\alpha}=\{0\}$. 
\end{defn}

 The crossed product of~$A$ by~$G$ under $V$, denoted by~$A\rtimes_GV$, is the universal $\Cst$\nb-algebra for covariant representations that vanish on all redundancies. Thus~$A\rtimes_GV$ is isomorphic to the quotient of~$\Toep(A,G,V)$ by the ideal generated by all redundancies. A covariant representation of~$(A,G,V)$ that vanishes on such an ideal was called \emph{strongly covariant} by Exel.  He was able to prove that~$A$ is embedded into~$A\rtimes_GV$. The crossed product carries a canonical~$G$\nb-grading, and a representation of $A\rtimes_GV$ is faithful on its fixed-point algebra if and only if it is faithful on~$A$.

 If~$P$ is a subsemigroup of~$G$, sometimes an action of~$P$ on a $\Cst$-algebra~$A$ may be enriched to an interaction group $(A,G,V)$ so that~$V_p=\alpha_p$ for all~$p\in P$. Under certain assumptions, $V$ is unique if it exists and $A\rtimes_GV$ is generated by~$A$ and isometries~$\{v_p\}_{p\in P}$ \cite[Theorem 12.3]{Exel:New_look}. We will see that if $P$ is reversible, in the sense that~$pP\cap qP\neq\emptyset$ and~$Pp\cap Pq\neq\emptyset$ for all $p,q\in P$, and~$G=P^{-1}P=PP^{-1}$, then $A\rtimes_GV$ can be obtained from a covariance algebra of a certain product system if $\{V_p\}_{p\in P}$ generates the image of~$G$ under~$V$. So we will assume that~$V$ is an interaction group which extends an action of $P$ by endomorphisms of~$A$ and~$V_{p^{-1}}\circ\alpha_p=\id_A$. This holds if and only if the \Star partial representation of~$G$ in~$A\rtimes_GV$ restricts to an isometric representation of~$P$.
 
 \begin{lem}\label{lem:isometry} Let $(i,s)$ denote the representation of $(A,G,V)$ in $A\rtimes_GV$. Then~$s_p$ is an isometry if and only if $V_{p^{-1}}\circ\alpha_p=\id_A$. 
 \begin{proof} Suppose that $V_{p^{-1}}\circ\alpha_p=\id_A$. Let us prove that $\widehat{s}_p^*\widehat{s}_p-1$ vanishes on $\widehat{\Hilm[M]}_{(p^{-1})}=j_V(A)\widehat{s}_{p^{-1}}j_V(A)$. Since $\widehat{s}$ is a \Star partial representation of~$G$, one has that $\widehat{s}_{p^{-1}}=\widehat{s}_p^*$. Put $\beta_1=(p^{-1},p)$ and $\beta_2=(e)$. So both $\beta_1$ and $\beta_2$ belong to $\Hilm[W]_{(p^{-1})}$ and hence $\widehat{s}_p^*\widehat{s}_p-1\in\widehat{\Hilm[Z]}^{\{e,p^{-1}\}}$. Thus all we must do is prove that~$$(\widehat{s}_p^*\widehat{s}_p-1)j_V(A)\widehat{s}_p^*j_V(A)=\{0\}.$$ To do so, let $a\in A$. Then $$\widehat{s}_p^*\widehat{s}_pj_V(a)\widehat{s}_p^*=\widehat{s}_p^*j_V(V_p(a))\widehat{s}_p\widehat{s}_p^*=j_V(V_{p^{-1}}(\alpha_p(a)))\widehat{s}_p^*\widehat{s}_p\widehat{s}_p^*=j_V(a)\widehat{s}_p^*.$$ This proves that $\widehat{s}_p^*\widehat{s}_p-1$ is a redundancy. Hence $s_p$ is an isometry in~$A\rtimes_GV$.
 
 Now assume that~$s_p$ is an isometry. For each~$a$ in $A$, $$i(a)=s_p^*s_pi(a)s_p^*s_p=i(V_{p^{-1}}(\alpha_p(a))).$$ This shows that~$V_{p^{-1}}\circ\alpha_p=\id_A$ because~$A$ is embedded into $A\rtimes_GV$. 
 \end{proof}
 
 \end{lem}
 
 Thus in order to build a product system over~$P$ so that it encodes the interaction group, we suppose that $V_{p^{-1}}\circ\alpha_p=\id_A$ for all~$p\in P$. It follows from \cite[Lemma 2.3]{Exel:New_look} that, for all $p,q\in P$, we have $$V_{q^{-1}}V_{p^{-1}}=V_{q^{-1}p^{-1}},\quad V_{p^{-1}q}=V_{p^{-1}}V_q.$$ Let us now describe the product system associated to~$V$. This is defined as in~\cite{Larsen:Crossed_abelian}. Here we do not require~$P$ to be abelian since we assume $V_{p^{-1}}\circ\alpha_p=\id_A$. We set~$\Hilm_p\coloneqq A$, endowed with the right action of~$A$ through~$a\cdot b\coloneqq a\alpha_p(b)$ and the $A$\nb-valued inner product~$\braket{a}{b}=V_{p^{-1}}(a^*b).$ This provides~$\Hilm_p$ with a structure of right Hilbert $A$\nb-module because~$V_{p^{-1}}(a^*a)=0\Leftrightarrow a=0$ and $V_{p^{-1}}(ab)=V_{p^{-1}}(a)V_{p^{-1}}(b)$ whenever~$b$ lies in the range of~$\alpha_p$. The \Star homomorphism $\varphi_p\colon A\rightarrow\Bound(\Hilm_p)$ is given by the multiplication in~$A$, so that~$\varphi_p(a)\cdot b=ab$ for all~$a\in A$, $b\in\Hilm_p$. The correspondence isomorphism~$\mu_{p,q}\colon\Hilm_p\otimes_A\Hilm_q\cong\Hilm_{pq}$ sends an elementary tensor $a\otimes b$ to~$a\alpha_p(b)$. Using that~$\alpha_p$ is an endomorphism of~$A$, we deduce that~$\mu_{p,q}$ preserves the bimodule structure. It is also surjective because~$\alpha_p$ is unital for all~$p\in P$.
 
 \begin{lem} $\Hilm=(\Hilm_p)_{p\in P}$ is a product system.
 \begin{proof} We will prove that~$\mu_{p,q}$ preserves the inner product and that the multiplication in~$\Hilm$ is associative.
 
 Let~$a_0,a_1,b_0,b_1\in A$. Then \begin{align*}\braket{a_0\otimes b_0}{a_1\otimes b_1}&=V_{q^{-1}}(b_0^*V_{p^{-1}}(a_0^*a_1)b_1)\\&=V_{q^{-1}}(V_{p^{-1}}(\alpha_p(b_0)^*)V_{p^{-1}}(a_0^*a_1)V_{p^{-1}}(\alpha_p(b_1)))\\&=V_{q^{-1}}(V_{p^{-1}}(\alpha_p(b_0)^*a_0^*a_1\alpha_p(b_1)))\\&=V_{(pq)^{-1}}(\alpha_p(b_0)^*a_0^*a_1\alpha_p(b_1))&\\&=\braket{\mu_{p,q}(a_0\otimes b_0)}{\mu_{p,q}(a_1\otimes b_1)}.
 \end{align*} This completes the proof that~$\mu_{p,q}$ is an isomorphism of correspondences for all $p,q\in P$. Now let~$s\in P$, $a\in\Hilm_p$, $b\in\Hilm_q$ and~$c\in\Hilm_s$. Then \begin{align*}(\mu_{pq,s}(\mu_{p,q}\otimes1))\big(a\otimes b\otimes c\big)&=a\alpha_p(b)\alpha_{pq}(c)=a\alpha_p(b\alpha_q(c))\\&=(\mu_{p,qs}(1\otimes\mu_{q,s}))\big(a\otimes b\otimes c\big).
 \end{align*}
 \end{proof} 
 \end{lem}
 
 \begin{lem}\label{lem:nondegenerate} There is a covariant representation of~$(A,G,V)$ in~$A\times_{\Hilm}P$. It sends $g=p^{-1}q$ to $v_g\coloneqq j_p(1_p)^*j_q(1_q)$ and~$a$ to~$j_e(a)$. Moreover, given a word~$\beta=(g_1,g_2,\ldots,g_n)$ in~$G$, the map~$a\mapsto j_e(a)v_{\beta}$ is injective, where~$v_{\beta}=v_{g_1}v_{g_2}\cdots v_{g_n}$.
 \begin{proof} We begin by proving that $j_p(1_p)^*j_q(1_q)=j_{p'}(1_{p'})^*j_{q'}(1_{q'})$ for all~$p,q,p',q'\in P$ such that~$p^{-1}q=p'^{-1}q'$. To do so, we use that~$P$ is also left reversible. We can find~$s\in P$ with $s\in(pP\cap qP)\cap(p'P\cap q'P).$ Since~$(A,G,V)$ is nondegenerate, $\Hilm$ is faithful and hence~$I_{r^{-1}(r\vee s)}=\{0\}$ for all~$r\in P$ such that~$r\not\in sP$. So $$\Hilm_{\{s\}}=\bigoplus_{\substack{r\in sP}}\Hilm_r.$$  Now given $r\in sP$, we write~$b_r$ for an element in $\Hilm_r$. We compute 
 \begin{align*}t_{\{s\}}\left(\widetilde{t}(1_p)^*\widetilde{t}(1_q)\right)(b_r)&=t_{\{s\}}\left(\widetilde{t}(1_p)^*)(\alpha_q(b_r)\otimes 1\right)\\&=V_{p^{-1}}(\alpha_q(b_r))=V_{p^{-1}q}(b_r)=V_{{p'}^{-1}q'}(b_r)\\&=t_{\{s\}}\left(\widetilde{t}(1_{p'})^*\widetilde{t}(1_{q'})\right)(b_r). 
 \end{align*}Therefore, $ j_p(1_p)^*j_q(1_q)=j_{p'}(1_{p'})^*j_{q'}(1_{q'})$ and the map $g=p^{-1}q\mapsto j_p(1_p)^*j_q(1_q)$ is well defined. This gives a partial representation of~$G$ in~$A\times_{\Hilm}P$ because~$V$ is a partial representation. Given~$g=p^{-1}q\in G$, $v_{g^{-1}}=j_q(1_q)^*j_p(1_p)=v_g^*$. So~$g\mapsto v_g$ indeed defines a \Star partial representation of~$G$.
 
 Let us prove that~$(j_e,v)$ is covariant. Take $g=p^{-1}q\in G$ and $a\in A$.  Again we use the assumption that~$P$ is left reversible and choose~$s\in pP\cap qP$. Thus it suffices to show that $$ t_{\{s\}}\left(\widetilde{t}(1_p)^*\widetilde{t}(1_q)\widetilde{t}(a) \widetilde{t}(1_q)^*\widetilde{t}(1_p)\right)=t_{\{s\}}\left(\widetilde{t}(V_g(a))\widetilde{t}(1_p)^*\widetilde{t}(1_q)\widetilde{t}(1_q)^*\widetilde{t}(1_p)\right)$$ on~$\Hilm_r$ for~$r\in sP$. Indeed, given~$b_r\in\Hilm_r$, one has \begin{align*}
  t_{\{s\}}\left(\widetilde{t}(1_p)^*\widetilde{t}(1_q)\widetilde{t}(a) \widetilde{t}(1_q)^*\widetilde{t}(1_p)\right)(b_r)&=V_{g}(aV_{g^{-1}}(b_r))=V_{g}(a)V_{g}(V_{g^{-1}}(b_r))\\&=t_{\{s\}}\left(\widetilde{t}(V_g(a))\widetilde{t}(1_p)^*\widetilde{t}(1_q)\widetilde{t}(1_q)^*\widetilde{t}(1_p)\right)(b_r),
 \end{align*} so that~$(j_e,v)$ is a covariant representation of~$(A,G,V)$. 
 
 Let $\beta=(g_1,\ldots,g_n)$ be a word in~$G$. In order to prove that the map $a\mapsto j_e(a)v_{\beta}$ is injective, take $s\in K_{\mu(\beta)^{-1}}$. That is, $$s\in P\cap g_n^{-1}P\cap(g_n^{-1} g_{n-1}^{-1})P\cdots \cap (g_n^{-1}g_{n-1}^{-1}\cdots g_1^{-1})P.$$ It exists because~$G=PP^{-1}$. Using that~$V_{g}$ is unital for all~$g\in G$, we deduce that $$j_e(a)v_{\beta}v_s=j_e(a)v_{\beta}j_s(1)=j_e(a)j_{\overset{.}{\beta}s}(V_{\beta}(1))=j_e(a)j_{\overset{.}{\beta}s}(1)=j_{\overset{.}{\beta}s}(a).$$ Since the representation of~$\Hilm$ in~$A\times_{\Hilm}P$ is injective, the right-hand side above is nonzero. This guarantees that~$a\mapsto j_e(a)v_{\beta}$ is an injective map. 
  \end{proof}  
 \end{lem}

  The following is the main result of this subsection.
  
  \begin{prop} Let $P$ be a subsemigroup of a group~$G$ with~$G=P^{-1}P=PP^{-1}$. Let~$(A,G,V)$ be a nondegenerate interaction group extending an action $\alpha\colon P\to\mathrm{End}(A)$ by unital and injective endomorphisms. Suppose, in addition, that $V_{p^{-1}}\circ\alpha_p=\id_A$ for all~$p\in P$. Then~$A\rtimes_G V$ is isomorphic to~$A\times_{\Hilm}P$, where~$\Hilm$ is the product system constructed out of~$V$.
  
  \begin{proof} We begin by proving that $(j_e,v)$ factors through $A\rtimes_GV$. The pair~$(j_e,v)$ induces a \Star homomorphism $\widehat{\phi}\colon\Toep(A,G,V)\rightarrow A\times_{\Hilm}P$. Lemma \ref{lem:nondegenerate} says that the map~$a\mapsto j_e(a)v_{\beta}$ is injective for each word~$\beta$ in~$G$. Hence~\cite[Proposition 10.5]{Exel:New_look} implies that~$\widehat{\phi}$ is injective on~$\widehat{\Hilm[M]}_{\alpha}$. In particular, if~$c\in \widehat{\Hilm[Z]}^{\mu(\alpha)}$ is an $\alpha$\nb-redundancy, $\widehat{\phi}(c)j_r(\Hilm_r)=\{0\}$ for all~$r\in K_{\mu(\alpha)}$ because $$j_r(a)=j_e(a)j_r(1)=j_e(a)v_{\alpha}j_{\overset{.}{\alpha}^{-1}r}(1)\in\widehat{\phi}\big(\widehat{M}_{\alpha}\big)j_{\overset{.}{\alpha}^{-1}r}(1)$$ for all~$a$ in $A$. So~$\widehat{\phi}(c)$ must be zero in~$A\times_{\Hilm}P$. This induces a \Star homomorphism $\phi\colon A\rtimes_GV\rightarrow A\times_{\Hilm}P$ that is faithful on~$A$ and preserves the $G$\nb-grading of~$A\rtimes_GV$. Proposition 4.6 of~\cite{Exel:New_look} says that~$\phi$ is also faithful on the fixed-point algebra of $A\rtimes_GV$.  Now by Lemma~\ref{lem:isometry}, $s_p$ is an isometry in~$A\rtimes_GV$ for all $p\in P$.  Moreover, \cite[Lemma 2.3]{Exel:New_look} says that~$s_ps_q=s_{pq}$ for all $p,q\in P$. Hence one can show that the maps~$\Hilm_p\ni 1_p\mapsto s_p$ and $a\mapsto i(a)$ give rise to a representation of~$\Hilm$. By applying the injectivity of~$\phi$ on the fibres and the usual argument that the induced \Star homomorphism $\Toep_{\Hilm}\rightarrow A\rtimes_GV$ preserves the $G$\nb-grading, we conclude that such a representation must factor through~$A\times_{\Hilm}P$. The resulting \Star homomorphism is the inverse of~$\phi$.  
  \end{proof}  
  \end{prop}
 
\begin{rem} Let~$P$ be a reversible cancellative semigroup and let~$G$ be its enveloping group. Let~$A$ be a unital $\Cst$\nb-algebra and let~$\alpha\colon P\rightarrow\mathrm{End}(A)$ be an action by injective endomorphisms. Given a not-necessarily nondegenerate interaction group~$(A,G,V)$  extending~$\alpha$ with~$V_{p^{-1}}\circ\alpha_p=\id_A$, the equality $V_{q^{-1}}V_{p^{-1}}=V_{q^{-1}p^{-1}}$ still holds by~\cite[Lemma 2.3]{Exel:New_look}. Hence one may build a product system as above by letting~$\Hilm_p\coloneqq A\alpha_p(1)$ and~$\mu_{p,q}(a\alpha_p(1)\otimes_Ab\alpha_q(1))\coloneqq a\alpha_p(b)\alpha_{pq}(1)$ (see \cite{Larsen:Crossed_abelian}). Thus the covariance algebra of such a product system may be viewed as the crossed product of~$A$ under~$V$, generalising Exel's construction to interaction groups satisfying~$V_{q^{-1}}V_{p^{-1}}=V_{q^{-1}p^{-1}}$ that are not necessarily nondegenerate. For instance, the product system built in the previous subsection fits into this setting, where~$V_{g}(\chi_{S})\coloneqq\chi_{gS\cap P}$ for all~$S\in\mathcal{J}$ and $g\in G.$ 
\end{rem}

\section*{Acknowledgements} This article is part of my PhD dissertation, written at the University of Göttingen, under the supervision of Ralf Meyer. I am grateful to Ralf Meyer for his guidance and for helpful discussions and comments on a preliminary version of this paper. This work was supported by CNPq (Brazil) through grant 248938/2013-4.

 \appendix
 
 \section{Topologically graded \texorpdfstring{$\mathrm{C^*}$}{C*}-algebras}\label{sec:topologically_graded}
 
 In this appendix we recall basic aspects on the topological grading obtained from a discrete coaction. We refer to~\cite{Exel:Partial_dynamical} for a careful treatment of graded $\Cst$\nb-algebras as well as their relationship to $\Cst$\nb-algebras associated to Fell bundles.

\begin{defn} Let~$B$ be a $\Cst$\nb-algebra and~$A$ a $\Cst$\nb-subalgebra of~$B$. A positive linear map~$E\colon B\to A$ is a \emph{conditional expectation} if~$E$ is contractive and idempotent, $E(a)=a$ for all~$a\in A$ and $E$ is an $A$\nb-bimodule map, that is, $E(a_1ba_2)=a_1E(b)a_2$ for all $b\in B$, $a_1,a_2\in A$. It is called \emph{faithful} if~$E(b^*b)=0$ implies~$b=0$.
\end{defn}

\begin{defn}
Let $B$ be a $\Cst$\nb-algebra and~$G$ a discrete group. Let~$\{B_g\}_{g\in G}$ be a collection of closed subspaces of~$B$. We say that $\{B_g\}_{g\in G}$ is a grading for~$B$ if, for all~$g,h\in G$, one has
\begin{enumerate}
 
\item[\rm (i)] $B_g^*=B_{g^{-1}}$,

\item[\rm (ii)]  $B_gB_h\subseteq B_{gh}$,

\item[\rm (iii)]  $\{B_g\}_{g\in G}$ is linearly independent and~$\bigoplus_{g\in G}B_g$ is a dense subspace of~$B$.
\end{enumerate}

We say that~$B$ is a \emph{graded $\Cst$\nb-algebra}.
\end{defn}

\begin{defn} A grading $\{B_g\}_{g\in G}$ for a $\Cst$\nb-algebra~$B$ is a \emph{topological grading} if there exists a conditional expectation $E\colon B\to B_e$ vanishing on~$B_g$ for all~$g\neq e$.
\end{defn}

The proof of the next proposition is taken from~\cite[Proposition 2.6]{Chi-Keung:discrete_coactions}.

\begin{prop}\label{prop:grading_from_coaction} Let~$(A, G,\delta)$ be a coaction.  Then~$A$ carries a topological $G$\nb-grading. The corresponding spectral subspace at~$g\in G$ is given by $$A_g=\{a\in A\mid \delta(a)=a\otimes u_g\}.$$
\end{prop}
\begin{proof} Clearly, $A_gA_h\subseteq A_{gh}$ and $A_g^*=A_{g^{-1}}$ because~$\delta$ is a \Star homomorphism. If $a=\sum_{\substack{i=1}}^n a_{g_i}=0$, then $$\delta(a)=\sum_{\substack{i=1}}^n a_{g_i}\otimes u_{g_1}=0$$ implies~$E_{g_i}(\delta(a))=a_{g_i}\otimes u_{g_i}=0$, where~$E_{g_i}=\id_A\otimes\chi_{g_i}$ and~$\chi_{g_i}$ denotes the contractive projection of~$\Cst(G)$ onto~$\CC u_{g_i}$. So~$a_{g_i}=0$ for all~$i\in\{1,\ldots,n\}$.

Given $a\in A$, it follows from the coaction identity that $$(\delta\otimes\id_G)E_g(\delta(a))=E_g(\delta(a))\otimes u_g.$$ This shows that $E_g(\delta(a))=a_g\otimes u_g$ for some~$a_g$ in~$A_g$. We claim that~$\bigoplus_{\substack{g\in G}}A_g$ is dense in~$A$. Indeed, let~$a\in A$. Since~$G$ is discrete, $\delta$ automatically satisfies $\delta(A)(1\otimes \Cst(G))=A\otimes \Cst(G)$ (see~\cite{Baaj-Skandalis:Hopf_KK}). So we may approximate~$a\otimes 1\approx\sum_{i=1}^n\delta(a_i)(1\otimes u_{g_i})$. In addition, $\id_A=(\id_A\otimes 1_G)\circ\delta$  by \cite[Lemma A.24]{Echterhoff-Kaliszewski-Quigg-Raeburn:Categorical}, where $1_G\colon G\to\CC$ is the homomorphism $g\mapsto 1$. Then \begin{align*}a=(\id_A\otimes 1_G)(a\otimes 1)&=(\id_A\otimes 1_G)E_e\big((a\otimes 1)\big)\\&\approx(\id_A\otimes 1_G)\bigg(\overset{n}{\sum_{\substack{i=1}}}E_e\big(\delta(a_i)(1\otimes u_{g_i})\big)\bigg)\\&=(\id_A\otimes 1_G)\bigg(\overset{n}{\sum_{\substack{i=1}}}E_{g_i^{-1}}(\delta(a_i))\bigg)\in \bigoplus_{g\in G}A_g.\end{align*} Now we see that $(\id_A\otimes 1_G)\circ E_g\circ\delta$ gives a contractive projection onto $A_g$ that vanishes on~$A_h$ for~$h\neq g$. Hence $\{A_g\}_{g\in G}$ is a topological grading for~$A$.
\end{proof}

If $(A, G,\delta)$ is a coaction, we refer to the corresponding spectral subspace at~$e$ as the \emph{fixed-point algebra} for~$\delta$. 

\begin{defn}[\cite{Echterhoff-Kaliszewski-Quigg-Raeburn:Categorical}*{Definition A.45}] Let~$(A,G,\delta)$ and $(B,G,\gamma)$ be coactions. We say that a \Star homomorphism $\psi\colon A\to B$ is \emph{$\delta\text{-}\gamma$ equivariant} if
$(\psi\otimes\id_G)\circ\delta=\gamma\circ\psi$.
\end{defn}

\begin{prop}\label{prop:induced_coaction} Let~$(A,G,\delta)$ be a coaction. Let $I\idealin A$ be an ideal satisfying~$I=\overline{\bigoplus}_{\substack{g\in G}}I\cap A_g$. Then there is a coaction $\delta_{A/I}\colon A/I\to A/I\otimes \Cst(G)$ such that $$(q\otimes\id_G)\circ\delta=\delta_{A/I}\circ q.$$ In particular, $q$ is a $\delta\text{-}\delta_{A/I}$ equivariant \Star homomorphism.
\end{prop}
\begin{proof} Given $q(a)\in A/I$, set $\delta_{A/I}(q(a))\coloneqq (q\otimes\id_G)(\delta(a))$. This vanishes on~$I$ because it vanishes on~$I\cap A_g$ for all~$g\in G$ and~$I$ is generated by its intersection with the spectral subspaces. It is also injective because~$(\id_A\otimes 1_G)\circ \delta=\id_A$ gives~$q(a)\in\ker\delta_{A/I}$ if and only if $a$ belongs to~$I$. This satisfies the coaction identity because~$\delta$ does so. That $\delta_{A/I}$ is a nondegenerate \Star homomorphism is then clear. 
\end{proof}

\end{document}